\documentclass[reqno,oneside,12pt]{amsart}

\usepackage{hyperref}
\usepackage[ansinew]{inputenc}	
\usepackage{graphicx}
\usepackage{color}%
\usepackage[numbers, square]{natbib}
\usepackage{mathrsfs}
\usepackage{bbm}
\usepackage{tikz}
\usepackage{mathptmx}
\usepackage{amsmath,amsthm,amsfonts,amssymb}
\usepackage{dsfont,extarrows,enumerate}
\usepackage{fullpage}
\usepackage{xcolor}

\allowdisplaybreaks[4]

\DeclareMathOperator*{\Tr}{Tr\,}

\numberwithin{equation}{section}

\makeatletter
\def\@tocline#1#2#3#4#5#6#7{\relax
  \ifnum #1>\c@tocdepth % then omit
  \else
    \par \addpenalty\@secpenalty\addvspace{#2}%
    \begingroup \hyphenpenalty\@M
    \@ifempty{#4}{%
      \@tempdima\csname r@tocindent\number#1\endcsname\relax
    }{%
      \@tempdima#4\relax
    }%
    \parindent\z@ \leftskip#3\relax \advance\leftskip\@tempdima\relax
    \rightskip\@pnumwidth plus4em \parfillskip-\@pnumwidth
    #5\leavevmode\hskip-\@tempdima
      \ifcase #1
       \or\or \hskip 1em \or \hskip 2em \else \hskip 3em \fi%
      #6\nobreak\relax
    \dotfill\hbox to\@pnumwidth{\@tocpagenum{#7}}\par
    \nobreak
    \endgroup
  \fi}
\makeatother

\theoremstyle{plain} \newtheorem{theorem}{Theorem}[section]
\theoremstyle{plain} \newtheorem{proposition}[theorem]{Proposition}
\theoremstyle{plain} 
\theoremstyle{plain} \newtheorem{lemma}[theorem]{Lemma}
\theoremstyle{plain} \newtheorem{corollary}[theorem]{Corollary}
\theoremstyle{definition} 
\theoremstyle{definition} 
\theoremstyle{remark} \newtheorem{remark}[theorem]{Remark}
\theoremstyle{remark} \newtheorem{example}[theorem]{Example}

\renewcommand{\P}{\mathbf P}
\newcommand{\Prob}[1]{\P\left\{#1\right\}}

\newcommand{\R}{\mathbb R}
\newcommand{\GG}{\mathbb G}
\newcommand{\HH}{\mathbb H}
\newcommand{\II}{\boldsymbol{I}}
\newcommand{\NN}{\mathbb N}

\newcommand{\TT}{\mathbb T}
\newcommand{\VV}{\mathbb V}
\newcommand{\XX}{\mathbb X}
\renewcommand{\SS}{\mathbb S}

\newcommand{\Sphere}{S^{d-1}}
\newcommand{\Sphereplus}{S^{d-1}_+}

\newcommand{\fC}{\mathfrak{C}}
\newcommand{\fF}{\mathfrak{F}}
\newcommand{\fK}{\mathfrak{K}}
\newcommand{\fL}{\mathfrak{L}}
\newcommand{\fX}{\mathfrak{X}}
\newcommand{\fY}{\mathfrak{Y}}
\newcommand{\fZ}{\mathfrak{Z}}
\newcommand{\fG}{\mathfrak{G}}
\newcommand{\fLoc}{\mathfrak{K}}
\newcommand{\salg}{\mathcal{F}}

\newcommand{\sA}{\mathcal{A}}
\newcommand{\sB}{\mathcal{B}}
\newcommand{\sC}{\mathcal{C}}
\newcommand{\sE}{\mathcal{E}}
\newcommand{\sF}{\mathcal F}

\newcommand{\sH}{\mathcal{H}}

\newcommand{\sK}{\mathcal K}

\newcommand{\sP}{\mathcal{P}}

\renewcommand{\subset}{\subseteq}

\newcommand{\one}{\mathbf{1}}

\newcommand{\od}{\overset{{\rm d}}{=}}

\newcommand{\dodn}{\overset{{\rm d}}\longrightarrow}

\newcommand{\toHn}{\overset{\scriptscriptstyle{d_{\rm H}}}{\longrightarrow}}
\newcommand{\toFn}{\overset{\scriptscriptstyle{\rm Fell}}{\longrightarrow}}

\newcommand{\Matr}[1][\phantom{0}]{{\mathsf{M}}_d^{\mathrm{#1}}}
\newcommand{\RM}{\R^d\times\Matr}

\DeclareMathOperator{\conv}{\mathrm{conv}\,}
\DeclareMathOperator{\cl}{\mathrm{cl}\,}
\newcommand{\gl}{\mathbb{GL}_d\,}
\newcommand{\so}{\mathbb{SO}_d\,}
\newcommand{\ortho}{\mathbb{O}_d\,}
\DeclareMathOperator{\Int}{Int}
\DeclareMathOperator{\Nor}{Nor}
\DeclareMathOperator{\diag}{diag}
\DeclareMathOperator{\dom}{dom}

\renewcommand{\epsilon}{\varepsilon}
\newcommand{\eps}{\epsilon}
\renewcommand{\emptyset}{\varnothing}

\usepackage{fancybox}
\newlength{\querylen}
\begin{document}
\title{Generalised convexity with respect to families of affine maps} %matrix groups}

\author{Zakhar Kabluchko}
\address{Institute of Mathematical Stochastics, Department of Mathematics and Computer Science, University of M\"{u}nster, Orl\'{e}ans-Ring 10, D-48149 M\"{u}nster, Germany}
\email{zakhar.kabluchko@uni-muenster.de}

\author{Alexander Marynych}
\address{Faculty of Computer Science and Cybernetics, Taras Shevchenko National University of Kyiv, Kyiv, Ukraine}
\email{marynych@unicyb.kiev.ua}

\author{Ilya Molchanov}
\address{Institute of Mathematical Statistics and Actuarial Science,
  University of Bern, Alpeneggstr. 22, 3012 Bern,  Switzerland}
\email{ilya.molchanov@stat.unibe.ch}

\date{\today}

\begin{abstract}
  The standard closed convex hull of a set is defined as the
  intersection of all images, under the action of a group of rigid
  motions, of a half-space containing the given set. In this paper we
  propose a generalisation of this classical notion, that we call a
  $(K,\HH)$-hull, and which is obtained from the above construction by
  replacing a half-space with some other closed convex subset $K$ of
  the Euclidean space, and a group of rigid motions by a subset $\HH$
  of the group of invertible affine transformations. The main focus is
  on the analysis of $(K,\HH)$-convex hulls of random samples from
  $K$.
\end{abstract}



\maketitle

\section{Introduction}
\label{sec:introduction}

Let $\HH$ be a nonempty subset of the product $\R^d\times \gl$, where
$\gl$ is a group of all invertible linear transformations of
$\R^d$. We regard elements of $\HH$ as invertible affine
transformations of $\R^d$ by identifying $(x,g)\in\HH$ with a mapping
\begin{displaymath}
%  \label{eq:aff_map}
  \R^d \ni y\mapsto g(y+x)\in\R^d,
\end{displaymath}
which first translates the argument $y$ by the vector $x$ and
then applies the linear transformation $g\in\GG$ to the translated
vector.

Fix a closed convex set $K$ in $\R^d$ which is distinct from the whole
space. For a given set $A\subseteq \R^d$, consider the set
$$
\conv_{K,\HH}(A):=\bigcap_{(x,g)\in\HH:\;A\subseteq g(K+x)} g(K+x),
$$
where $g(B):=\{gz:z\in B\}$ and $B+x:=\{z+x:z\in B\}$, for $g\in\GG$,
$x\in\R^d$, and $B\subseteq\R^d$. If there is no $(x,g)\in\HH$ such
that $g(K+x)$ contains $A$, the intersection on the right-hand side is
taken over the empty family and we stipulate that
$\conv_{K,\HH}(A)=\R^d$. The set $\conv_{K,\HH}(A)$ is, by definition,
the intersection of all images of $K$ under affine transformations
from $\HH$, which contain $A$. We call the set $\conv_{K,\HH}(A)$,
which is easily seen to be closed and convex, the $(K,\HH)$-hull of
$A$; the set $A$ is said to be $(K,\HH)$-convex if it coincides with
its $(K,\HH)$-hull. This definition of the convex hull fits the
abstract convexity concept described in \cite{MR3767739}. It also
appeared in \cite{lac10}, where a generalised envelope of $A$ is
defined as the intersection of all sets containing $A$ from a certain
family.

In what follows we assume that $\HH$ contains the pair $(0,\II)$,
where $\II$ is the unit matrix. If $A\subset K$, then
\begin{equation}
  \label{eq:conv_KH_in_K}
  \conv_{K,\HH}(A)\subset K.
\end{equation}
It is obvious that a larger family $\HH$ results in a smaller
$(K,\HH)$-hull, that is, if $\HH\subset \HH_1$, then
$$
\conv_{K,\HH_1}(A)\subseteq \conv_{K,\HH}(A),\quad A\subset \R^d.
$$
In particular, if $\HH=\{0\}\times \{\II\}$ and $A\subset K$, then
$$
\conv_{K,\{0\}\times \{\II\}}(A)=K,
$$
so equality in \eqref{eq:conv_KH_in_K} is possible.  Since
$(K,\HH)$-hull is always a closed convex set which contains $A$,
$$
\conv(A)\subseteq \conv_{K,\HH}(A),\quad A\subset \R^d,
$$
where $\conv$ denotes the operation of taking the conventional closed
convex hull. If $K$ is a fixed closed half-space,
$\HH=\R^d\times \so$, where $\so$ is the special orthogonal group,
then $\conv_{K,\HH}(A)=\conv(A)$, since every closed half-space can be
obtained as an image of a fixed closed half-space under a rigid
motion.  The equality here can also be achieved for an arbitrary closed
convex set $K$ and all bounded $A\subset \R^d$ upon letting
$\HH=\R^d\times \gl$, see Proposition~\ref{prop:largest_group} below.

A nontrivial example, when $\conv_{K,\HH}(A)$ differs both from
$\conv(A)$ and from $K$, is as follows. If $K$ is a closed ball of a
fixed radius and $\HH=\R^d\times\{\II\}$, then $\conv_{K,\HH}(A)$ is
known in the literature as the ball hull of $A$, see \cite{Bezdek_et_al:2007,Fodor+Kevei+Vigh:2014,Paouris+Pivovarov:2017}, and, more generally,
if $K$ is an arbitrary convex body (that is, a compact convex set with
nonempty interior) and $\HH=\R^d\times \{\II\}$, then
$\conv_{K,\HH}(A)$ is called the $K$-hull of $A$, see
\cite{Fodor+Papvari+Vigh:2020,jah:mar:ric17,MarMol:2021}. It is also clear from the
definition that further nontrivial examples could be obtained from
this case by enlarging the family of linear transformations involved
in $\HH$.

In the above examples the set $\HH$ takes the form $\TT\times\GG$ for
some $\TT\subset \R^d$ and $\GG\subset \gl$. We implicitly assume
this, whenever we specify $\TT$ and $\GG$. Furthermore, many
interesting examples arise if $\TT$ is a linear subspace of $\R^d$ and
$\GG$ is a subgroup of $\gl$.

The paper is organised as follows. In Section~\ref{sec:KH_hulls} we
analyse some basic properties of $(K,\HH)$-hulls and show how various
known hulls in {\it convex geometry} can be obtained as particular
cases of our construction. However, the main focus in our paper will
be put on probabilistic aspects of $(K,\HH)$-hulls. As in many other
models of {\it stochastic geometry}, we shall study $(K,\HH)$-hulls of
random samples from $K$ as the size of the sample tends to infinity.
In Section~\ref{sec:random} we introduce a random closed set which can
be thought of as a variant of the Minkowski difference between the set
$K$ and the $(K,\HH)$-hull of a random sample from $K$. The limit
theorems for this object are formulated and proved in Section
\ref{sec:poiss-point-proc}. Various properties of the limiting random
closed set are studied in Section~\ref{sec:zero-cells-matrix}. A
number of examples for various choices of $K$ and $\HH$ are presented
in Section~\ref{sec:examples}. Some technical results related mostly
to the convergence in distribution of random closed sets are collected
in the Appendix.

\section{\texorpdfstring{$(K,\HH)$}{(K,H)}-hulls of subsets of \texorpdfstring{$\R^d$}{Rd}}\label{sec:KH_hulls}

It is easy to see that $\conv_{K,\HH}(K)=K$ and $\conv_{K,\HH}(A)$ is
equal to the intersection of all $(K,\HH)$-convex sets containing $A$.
We now show that the $(K,\HH)$-hull is an idempotent operation.

\begin{proposition}
  If $A\subset K$, then
  \begin{displaymath}
    \conv_{K,\HH}\big(\conv_{K,\HH}(A)\big)=\conv_{K,\HH}(A).
  \end{displaymath}
\end{proposition}
\begin{proof}
  We only need to show that the left-hand side is a subset of the
  right-hand one. Assume that $z$ belongs to the left-hand side, and
  $z\notin g(K+x)$ for at least one $(x,g)\in\HH$ such that
  $A\subset g(K+x)$. The latter implies $\conv_{K,\HH}(A)\subset g(K+x)$, so that
  $\conv_{K,\HH}(\conv_{K,\HH}(A))$ is also a subset of $g(K+x)$,
  which is a contradiction. 
\end{proof}

For each $A\subset \R^d$, denote
\begin{displaymath}
  K\ominus_{K,\HH} A:=\big\{(x,g)\in\HH: A\subset g(K+x)\big\}. 
\end{displaymath}
If $\HH=\R^d\times\{\II\}$, the set
$\big\{x\in\R^d: (-x,\II)\in (K\ominus_{K,\HH} A)\big\}$ is the usual
Minkowski difference
\begin{displaymath}
  K\ominus A:=\{x\in\R^d: A+x\subseteq K\}
\end{displaymath}
of $K$ and $A$, see page~146 in~\cite{schn2}. By the definition of $(K,\HH)$-hulls 
\begin{displaymath}
  \conv_{K,\HH}(A)=\bigcap_{(x,g)\in K\ominus_{K,\HH} A} g(K+x),
\end{displaymath}
and, therefore, $A$ is $(K,\HH)$-convex if and only if
\begin{displaymath}
  A=\bigcap_{(x,g)\in K\ominus_{K,\HH} A} g(K+x). 
\end{displaymath}

The following is a counterpart of Proposition 2.2 in \cite{MarMol:2021}.

\begin{lemma}\label{lem:minkowski}
  For every $A\subset \R^d$, we have
  $$
  K\ominus_{K,\HH} A=K\ominus_{K,\HH} \big(\conv_{K,\HH}(A)\big).
  $$
\end{lemma}
\begin{proof}
  Since $A\subset \conv_{K,\HH}(A)$, the right-hand side is a subset
  of the left-hand one. Let $(x,g)\in K\ominus_{K,\HH} A$. Then
  $A\subseteq g(K+x)$, and, therefore,
  $\conv_{K,\HH}(A)\subseteq g(K+x)$. The latter means that
  $(x,g)\in K\ominus_{K,\HH} (\conv_{K,\HH}(A))$.
\end{proof}

Now we shall investigate how $\conv_{K,\HH}(A)$ looks for various
choice of $K$ and $\HH$, in particular, how various known hulls
(conventional, spherical, conical, etc.) can be derived as particular
cases of our model. In order to proceed, we recall some basic notions
of convex geometry. Let $K$ be a closed convex set, and let
\begin{displaymath}
  h(K,u):=\sup\big\{\langle x,u\rangle : x\in K\big\},\quad u\in\R^d
\end{displaymath}
denote the support function of $K$ in the direction $u$, where
$\langle x,u\rangle$ is the scalar product.  Put
\begin{displaymath}
  \dom(K):=\big\{u\in\R^d : h(K,u)<\infty\big\}
\end{displaymath}
and note that $\dom(K)=\R^d$ for compact sets $K$. The cone $\dom(L)$
is sometimes called the barrier cone of $L$, see the end of Section~2 in
\cite{Rockafellar:1972}. For $u\in \dom(K)$, let $H(K,u)$,
$H^{-}(K,u)$ and $F(K,u)$ denote the support plane, supporting
halfspace and support set of $K$ with outer normal vector $u\neq0$,
respectively. Formally,
\begin{displaymath}
  H(K,u):=\big\{x\in\R^d:\langle x,u\rangle = h(K,u)\big\},\quad
  H^{-}(K,u):=\big\{x\in\R^d:\langle x,u\rangle \leq h(K,u)\big\}
\end{displaymath}
and $F(K,u)=H(K,u)\cap K$.  We shall also need notions of supporting
and normal cones of $K$ at a point $v\in K$. The supporting cone at
$v\in K$ is defined by 
\begin{equation}
  \label{eq:supporting-cone}
  S(K,v):=\cl\Big(\bigcup_{\lambda>0}\lambda(K-v)\Big),
\end{equation}
where $\cl$ is the topological closure, see page~81 in~\cite{schn2}. If
$v\in F(K,u)$ for some $u\in \dom(K)$, then
\begin{equation}
  \label{eq:cone_in_halfspace}
  S(K,v)+v\subseteq H^{-}(K,u).
\end{equation}

For $v$ which belong to the boundary $\partial K$ of $K$, the normal cone
$N(K,v)$ to $K$ at $v$ is defined by
$$
N(K,v):=\big\{u\in\R^d\setminus\{0\}:v\in H(K,u)\big\}\cup\{0\}.
$$

\begin{proposition}
  \label{prop:cone}
  Let $K$ be a closed convex set, and let $\HH=\TT\times\GG$, where
  $\TT=\R^d$, and $\GG=\{\lambda\II:\lambda>0\}$ is the group of all
  scaling transformations. If $A\subset K$, then
  \begin{equation}
    \label{eq:7}
    \conv_{K,\HH}(A)=\bigcap_{x\in\R^d,v\in\partial K, A\subset
      S(K,v)+v+x} \big(S(K,v)+v+x\big),
  \end{equation}
  that is, $\conv_{K,\HH}(A)$ is the intersection of all translations
  of supporting cones to $K$ that contain $A$.
\end{proposition}
\begin{proof}
  If $A\subset \lambda K+x$, then $A\subset S(K,v)+v+x$ for any
  $v\in K$. Hence, we only need to show that the right-hand side of
  \eqref{eq:7} is contained in the left-hand one. Assume that
  $z\in S(K,v)+v+x$ for all $v\in\partial K$ and $x\in\R^d$ such that
  $A\subset S(K,v)+v+x$, but $z\notin \lambda_0 K+y_0$ for some
  $\lambda_0>0$ and $y_0\in\R^d$ with $A\subset \lambda_0 K+y_0$. By
  the separating hyperplane theorem, see Theorem~1.3.4 in
  \cite{schn2}, there exists a hyperplane $H_0\subset \R^d$ such that
  $\lambda_0 K+y_0\subset H^{-}_0$ and $z\in H_0^{+}$, where
  $H_0^{\pm}$ are the open half-spaces bounded by $H_0$. Let $u_0$ be
  the unit outer normal vector to $H_0^-$, and note that
  $u_0\in \dom(K)$. Choose an arbitrary $v_0$ from the support
  set $F(K,u_0)$. Since
  $\lambda_0 K=\lambda_0 (K-v_0)+\lambda_0 v_0\subset
  S(K,v_0)+\lambda_0 v_0$, we have
  $A\subset S(K,v_0)+\lambda_0 v_0+y_0$. However,
  \begin{align*}
    S(K,v_0)+\lambda_0 v_0+y_0
    &=S(\lambda_0 K,\lambda_0 v_0)+\lambda_0 v_0+y_0\\
    &\subseteq H^{-}(\lambda_0 K_0,u_0)+y_0
      =H^{-}(\lambda_0 K_0+y_0,u_0)\subseteq H_0^{-}\cup H_0,
  \end{align*} 
  where the penultimate inclusion follows from
  \eqref{eq:cone_in_halfspace}. Thus, $z\notin S(K,v_0)+v_0+x_0$ with
  $x_0=(\lambda_0-1)v_0+y_0$, and $S(K,v_0)+v_0+x_0$ contains $A$. The
  obtained contradiction completes the proof.
\end{proof}

\begin{corollary}
  \label{cor:smooth}
  If $K$ is a smooth convex body (meaning that the normal cone at each
  boundary point is one-dimensional), $\TT=\R^d$, and 
  $\GG=\{\lambda\II:\lambda>0\}$ is the group of all
  scaling transformations, then  $\conv_{K,\HH}(A)=\conv(A)$, for all $A\subset K$.
\end{corollary}
\begin{proof}
  Since $K$ is a smooth convex body, $\dom(K)=\R^d$ and its
  supporting cone at each boundary point is equal to the supporting
  half-space. The convex hull of $A$ is exactly the intersection of
  all such half-spaces.
\end{proof}

The next result formalises an intuitively obvious fact that the
$(K,\HH)$-hull of a bounded set $A$ coincides with $\conv(A)$ for
an arbitrary $K$ provided $\HH$ is sufficiently rich, in particular, if
$\HH=\R^d\times\gl$.

\begin{proposition}\label{prop:largest_group}
  Let $K$ be a closed convex set with nonempty interior, 
  $\TT=\R^d$, and let $\GG$ be the
  group of all scaling and orthogonal transformations of $\R^d$, that
  is,
  $$
  \GG=\big\{x\mapsto \lambda g(x) : \lambda>0,\; g\in\so\big\},
  $$
  where $\so$ is the special orthogonal group of $\R^d$. Then
  $\conv_{K,\HH}(A)=\conv(A)$ for all bounded $A\subset \R^d$.
\end{proposition}
\begin{proof}
  It is clear that $\conv(A)\subset \conv_{K,\HH}(A)$. In
  the following we prove the opposite inclusion. Since
  $A\subseteq g(K+x)$ if and only if $\conv(A)\subseteq g(K+x)$, we
  have $\conv_{K,\HH}(A) = \conv_{K,\HH}(\conv(A))$ and
  there is no loss of generality in assuming that $A$ is compact and
  convex. Take a point
  $z\in \R^d\setminus A$. We need to show that there exists a pair 
  $(x,g)\in \R^d\times \GG$ (depending on $z$) such that
  $z\notin g(K+x)$ and $A\subset g(K+x)$. By the separating hyperplane
  theorem, see Theorem~1.3.4 in \cite{schn2}, there exists a
  hyperplane $H\subset \R^d$ such that $A\subset H^{-}$ and
  $z\in H^{+}$. If $K$ is compact, Theorem~2.2.5 in~\cite{schn2}
  implies that the boundary of $K$ contains at least one point at
  which the supporting cone is a closed half-space. This holds
  also for each closed convex $K$, which is not necessarily bounded, by
  taking intersections of $K$ with a growing family of closed
  Euclidean balls. 
  Let $v\in \partial K$ be such a point. After
  applying appropriate translation $x_0$ and orthogonal transformation
  $g_0\in\so$, we may assume that the supporting cone
  $S(g_0 (K+x_0),g_0 (v+x_0))$ is the closure of $H^{-}$. Thus,
  $$
  A\subseteq \bigcup_{\lambda>0}\lambda \big(g_0 (K-v)\big)\quad
  \text{and}\quad z\notin \bigcup_{\lambda>0}\lambda \big(g_0 (K-v)\big).
  $$
  It remains to show that there exists a $\lambda_0>0$ such that
  $A\subseteq \lambda_0 (g_0 (K-v))$. Assume that for every $n\in\NN$
  there exists an $a_{n}\in A$ such that $a_n\notin n(g_0 (K-v))$. Since
  $A$ is compact, there is a subsequence $(a_{n_j})$ converging to
  $a\in A\subset H^{-}$ as $j\to\infty$. Thus, there exists a
  $\lambda_0>0$ such that $a$ lies in the interior of
  $\lambda_0(g_0 (K-v))$. Hence, $a_{n_j}\in \lambda_0(g_0 (K-v))$ for
  all sufficiently large $j$, which is a contradiction.
\end{proof}  

\begin{remark}
  For unbounded sets $A$ the claim of
  Proposition~\ref{prop:largest_group} is false in general. As an
  example, one can take $d=2$, $A$ is a closed half-space and $K$ is
  an acute closed wedge. Thus, $\conv_{K,\HH}(A)=\R^2$, whereas
  $\conv(A)=A$.
\end{remark}

\begin{proposition}\label{eq:ellipses}
  Let $\GG=\gl$ be the general linear group,
  $\TT=\{0\}$, and let $K=B_1$ be the unit ball in $\R^d$.  Then, for
  an arbitrary compact set $A\subset \R^d$, it holds
  $\conv_{K,\HH}(A) = \conv(A\cup \check{A})$, where
  $\check{A}=\{-z:z\in A\}$. 
\end{proposition}
\begin{proof}
  The images of the unit ball under the elements of $\gl$ are all
  ellipsoids centered at $0$. Since each of these ellipsoids
  is origin symmetric and convex, it is clear that
  $\conv_{K,\HH}(A) \supseteq \conv(A\cup \check{A})$. Let us prove the
  converse inclusion. Since replacing $A$ by the compact convex set
  $\conv(A\cup \check{A})$ does not change its $(K,\HH)$-hull, it suffices
  to assume that $A$ is an origin symmetric compact convex set. Let us
  take some $z\notin A$. We need to construct an ellipsoid $E$
  centered at the origin and such that $A\subset E$, whereas
  $z\notin E$. By the separating hyperplane theorem, see Theorem~1.3.4
  in \cite{schn2}, there exists an affine hyperplane $H\subset \R^d$
  such that $A\subset H^{-}$ and $z\in H^{+}$, where $H^{\pm}$ are
  open half-spaces bounded by $H$. Let $x=(x_1,\ldots, x_d)$ be the
  coordinate representation of a generic point in $\R^d$. After
  applying an orthogonal transformation, we may assume that the
  hyperplane $H$ is $\{x_1 = a\}$ for some $a>0$. Then,
  $A\subset \{|x_1|<a\}$, while $z\in \{x_1>a\}$. Hence,
  $A \subset \{x\in \R^d: |x_1|\leq a-\eps, x_2^2+\cdots + x_d^2 \leq
  R^2\}=:D$ for sufficiently small $\eps>0$ and sufficiently large
  $R>0$.  Clearly, there is an ellipsoid $E$ centered at $0$,
  containing $D$ and contained in the strip $\{|x_1|<a\}$. By
  construction, we have $A\subset E$ and $z\notin E$, and the proof is
  complete.
\end{proof}

Our next example deals with conical hulls.

\begin{proposition}\label{prop:conic_hull}
  Let $\TT=\{0\}$, $\GG=\so$ be the special orthogonal group, and let
  $K$ be the closed half-space in $\R^d$ such that $0\in\partial
  K$. If $A\subset K$, then
  $$
  \conv_{K,\HH}(A)=\cl({\rm pos}(A)),
  $$
  where
  $$
  {\rm pos}(A):=\Big\{\sum_{i=1}^{m}\alpha_i x_i
    :\;\alpha_i\geq 0,\;x_i\in A,m\in\NN\Big\}
  $$
  is the positive (or conical) hull of $A$.
\end{proposition}
\begin{proof}
  By definition, $\conv_{K,\HH}(A)$ is the intersection of all closed
  half-spaces which contain the origin on the boundary, because every
  such half-space is an image of $K$ under some orthogonal
  transformation. Since $\cl({\rm pos}(A))$ is the intersection
  of {\it all} closed convex cones which contain $A$,
  $\cl({\rm pos}(A))\subset \conv_{K,\HH}(A)$. On the other hand,
  every closed convex cone is the intersection of its supporting
  half-spaces, see Corollary~1.3.5 in \cite{schn2}. Since all these
  supporting half-spaces contain the origin on the boundary, 
  $\cl({\rm pos}(A))$ is the intersection of {\it some}
  family of half-spaces containing the origin on the boundary, which
  means that $\conv_{K,\HH}(A)\subset \cl({\rm pos}(A))$.
\end{proof}

The next corollary establishes connections with a probabilistic model
studied recently in \cite{kab:mar:tem:19}. We shall return to this
model in Section~\ref{sec:examples}. Let 
\begin{equation}
  \label{eq:upper-ball}
  B_1^{+}:=\big\{(x_1,x_2,\ldots,x_d):
  x_1^2+\cdots+x_d^2\leq 1,x_1\geq 0\big\}
\end{equation}
be the unit upper half-ball in $\R^{d}$, and let
\begin{displaymath}
  \Sphereplus :=\big\{(x_1,x_2,\ldots,x_d):x_1^2+\cdots+x_d^2 =
  1,x_1\geq 0\big\}
\end{displaymath}
be the unit upper $(d-1)$-dimensional half-sphere.  Further, let
$\pi:\R^d\setminus\{0\}\to \Sphere$ be the mapping
$\pi(x)=x/\|x\|$, where $\Sphere$ is the unit sphere.

\begin{corollary}
  Let $K=B_1^{+}$, $\GG=\so$ and $\TT=\{0\}$. Then, for an arbitrary
  $A\subset K$, it holds
  \begin{equation}\label{eq:cones1}
    \conv_{K,\HH}(A)=\cl({\rm pos}(A))\cap B_1.
  \end{equation}
 Furthermore, if $A\neq\{0\}$, then
  \begin{equation}\label{eq:cones2}
    \conv_{K,\HH}(A)\cap\Sphere=\cl\big({\rm pos}(A)\big)\cap \Sphere
    =\cl\big({\rm pos}(\pi(A\setminus\{0\}))\big)\cap \Sphere,
  \end{equation}
  which is the closed spherical hull of the set
  $\pi(A\setminus\{0\})\subset \Sphere$.
\end{corollary}
\begin{proof}
  Note that $g(B_1^{+})=g(H_0^{+})\cap B_1$ for every $g\in\so$,
  where 
  \begin{equation}
    \label{eq:upper-half}
    H_0^+:=\big\{(x_1,x_2,\ldots,x_d)\in\R^d : x_1\geq 0\big\}.
  \end{equation}  
  Thus,
  \begin{align*}
    \conv_{K,\HH}(A)&=\bigcap_{g\in\so,A\subset g(B_1^{+})}g(B_1^{+})\\
    &=\bigcap_{g\in\so,A\subset g(H_0^{+})}
    \left(g(H_0^{+})\cap B_1\right)
    =\bigg(\bigcap_{g\in\so,A\subset g(H_0^{+})}g(H_0^{+})\bigg)\cap B_1,
  \end{align*}
  where we have used that $A\subset B_1$. By
  Proposition~\ref{prop:conic_hull}, the right-hand side is
  $\cl({\rm pos}(A))\cap B_1$. The first equation in \eqref{eq:cones2}
  is a direct consequence of \eqref{eq:cones1}, while the second one
  follows from
  ${\rm pos}(A)={\rm pos}(A\setminus\{0\})={\rm
    pos}(\pi(A\setminus\{0\}))$.
\end{proof}

\section{\texorpdfstring{$(K,\HH)$}{(K,H)}-hulls of random samples from \texorpdfstring{$K$}{K}}\label{sec:random}

From now on we additionally assume that $K\in\sK^d_{(0)}$, that
is, $K$ is a compact convex set in $\R^d$ which contains the origin
in its interior. Fix a complete probability space $(\Omega,\salg,\P)$. For
$n\in\NN$, let $\Xi_n:=\{\xi_1,\xi_2,\ldots,\xi_n\}$ be a sample of
$n$ independent copies of a random variable $\xi$ uniformly
distributed on $K$. Put
\begin{displaymath}
  Q_n:=\conv_{K,\HH}(\Xi_n)
\end{displaymath}
and
\begin{equation}
  \label{eq:4}
  \XX_{K,\HH}(\Xi_n):=\big\{(x,g)\in\HH:\Xi_n\subseteq g(K+x)\big\}
  =K\ominus_{K,\HH} \Xi_n=K\ominus_{K,\HH} Q_n,
\end{equation}
where the last equality follows from Lemma~\ref{lem:minkowski}.

We start with a simple lemma which shows that, for every $n\in\NN$,
$\XX_{K,\HH}(\Xi_n)$ is a random closed subset of $\HH$ equipped with
the relative topology induced by $\RM$, see the Appendix for the
definition of a random closed set. Here and in what follows $\Matr$
denotes the set of $d\times d$ matrices with real entries.  Note that
$Q_n$ is closed, being the intersection of closed sets.

\begin{lemma}
  \label{lemma:XXn}
  For all $n\in\NN$, $\XX_{K,\HH}(\Xi_n)$ is a random closed set
  in $\HH$.
\end{lemma}
\begin{proof}
  Let $\XX_\xi:=\{(x,g)\in\HH:\xi\in g(K+x)\}$.  For each compact set
  $L\subset\HH$, we have
  \begin{equation}\label{eq:is_random_set}
    \big\{\omega\in\Omega:\;\XX_{\xi(\omega)}\cap L\neq\emptyset \big\}
    =\big\{\omega\in\Omega:\;\xi(\omega)\in LK \big\},
%    =\{\omega\in\Omega:\;H\xi(\omega)\cap K\neq\varnothing\},
  \end{equation}
  where $LK:=\{g(z+x):(x,g)\in L, z\in K\}$. Note that $LK$ is a
  compact set, hence it is Borel, and the event on the right-hand side
  of \eqref{eq:is_random_set} is measurable. Thus, in view of
  \eqref{eq:is_random_set}, $\XX_\xi$ is a random closed set in the
  sense of Definition~1.1.1 in \cite{mo1}. Hence,
  \begin{displaymath}
    \XX_{K,\HH}(\Xi_n)=\XX_{\xi_1}\cap\cdots\cap \XX_{\xi_n}
  \end{displaymath}
  is also a random closed set, being a finite intersection of random
  closed sets, see Theorem~1.3.25 on \cite{mo1}. 
\end{proof}

We are interested in the asymptotic properties of 
$\XX_{K,\HH}(\Xi_n)$ as $n\to\infty$.  Note that the sequence of sets
$(Q_n)$ is increasing in $n$ and, for every $n\in\NN$,
$P_n:=\conv(\Xi_n)\subseteq Q_n$. Since $P_n$ converges almost surely
to $K$ in the Hausdorff metric as $n\to\infty$, the sequence $(Q_n)$
also converges almost surely to $K$. Since the sequence of sets
$(\XX_{K,\HH}(\Xi_n))$ is decreasing in $n$, 
\begin{equation}
  \label{eq:x_n_converges_to_k_minus_k}
  \XX_{K,\HH}(\Xi_n) \;\downarrow\; (K\ominus_{K,\HH}
  K)=\big\{(x,g)\in\HH: K\subseteq g(K+x) \big\}
  \quad \text{a.s. as}\quad n\to\infty.
\end{equation}
Since we assume $(0,\II)\in \HH$, the set $K\ominus_{K,\HH} K$
contains $(0,\II)$. However, the set $K\ominus_{K,\HH} K$ may contain
other points, e.g., all $(0,g)\in\HH$ such that $K\subset gK$.

It is natural to ask whether it is possible to renormalise, in an
appropriate sense, the set $\XX_{K,\HH}(\Xi_n)$ such that it would
converge to a random limit? Before giving a rigorous answer to this
question we find it more instructive to explain our approach
informally. While doing this, we shall also recollect necessary
concepts, and introduce some further notation.

First of all, note that 
$$
\XX_{K,\HH}(\Xi_n)=\XX_{K,\R^d\times \gl}(\Xi_n)\cap
\HH\quad\text{and}\quad K\ominus_{K,\HH} K=(K\ominus_{K,\R^d\times \gl} K)\cap \HH.
$$
Thus, we can first focus on the special case $\HH=\R^d\times \gl$
and then derive the corresponding result for an arbitrary $\HH$ by taking
intersections. Denote
\begin{displaymath}
  \XX_n:=\XX_{K,\R^d\times \gl}(\Xi_n). 
\end{displaymath}

In order to quantify the convergence in
\eqref{eq:x_n_converges_to_k_minus_k} and derive a meaningful limit
theorem for $\XX_n$, we shall pass to tangent spaces. The vector space $\Matr$ is 
a tangent space to the Lie group $\gl$ at $\II$ and
is the Lie algebra of $\gl$. However, for our purposes the
multiplicative structure of the Lie algebra is not needed and we
use only its linear structure as of a vector space over $\R$. Let
$$
\exp: \Matr\to \gl
$$
be the standard matrix exponent, and let $\VV$ be a sufficiently
small neighbourhood of $\II$ in $\gl$, where the exponent is
bijective, see, for example, Theorem 2.8 in \cite{Hall:2003}. Finally,
let $\log:\VV\to \Matr$ be its inverse and
define mappings
$\widetilde{\log}: \R^d \times \VV \to \R^d \times \log\VV$ and
$\widetilde{\exp}: \R^d \times \log \VV \to \R^d \times \VV$ by
$$
\widetilde{\log}(x,g)=(x,\log g),
\quad \widetilde{\exp}(x,h)=(x,\exp h),
\quad x\in\R^d,\; g\in\VV,\; h\in \log\VV.
$$
Using the above notation, we can write
\begin{multline*}
  \widetilde{\log}\big(\XX_n\cap (\R^d \times \VV)\big)
  =\big\{(x,C)\in \R^d\times \log \VV : \Xi_n\subseteq \exp(C)(K+x)\big\}\\
  =\big\{(x,C)\in \RM:  \Xi_n\subseteq \exp(C)(K+x)\big\}\cap (\R^d\times \log \VV)
  =\fX_n \cap (\R^d\times \log \VV),
\end{multline*}
where we set
$$
  \fX_n:=\big\{(x,C)\in\RM:\Xi_n\subseteq \exp(C)(K+x) \big\}.
$$
In the definition of $\fX_n$ the space $\RM$ should be regarded as a
tangent vector space at $(0,\II)$ to the Lie group of all invertible
affine transformations of $\R^d$. Similarly to Lemma~\ref{lemma:XXn},
it is easy to see that $\fX_n$ is a random closed set in $\RM$. Note
that $\fX_n$ may be unbounded (in the product of the standard norm on
$\R^d$ and some matrix norm on $\Matr$) and, in general, is not
convex. 

We shall prove below, see Theorem~\ref{thm:main1}, that the sequence
$(n\fX_n)$ converges in distribution to a nondegenerate random set
$\check{\fZ}_K=\{-z:z\in\fZ_K\}$ as random closed sets, see the
Appendix for necessary formalities. We pass from the random set
$\fZ_K$ defined at \eqref{thm:main1:claim} to its reflected variant to
simplify later notation.  Moreover, for an arbitrary compact convex
subset $\fK$ in $\RM$ which contains the origin, the sequence of
random sets $(n\fX_n\cap \fK)$ converges in distribution to
$\check{\fZ}_K\cap \fK$ on the space of compact subsets of $\RM$
endowed with the usual Hausdorff metric.

Since $\VV$ contains the origin in its interior, there exists an
$n_0\in\NN$ such that
\begin{equation}\label{eq:compact_inside_neihgborhood}
  n(\R^d\times \log \VV)\supseteq \fK
  \quad\text{and}\quad
  \R^d\times \VV\supseteq \widetilde{\exp}(\fK/n),\quad n\geq n_0.
\end{equation}
Hence,
$$
n\fX_n\cap \fK=n\big(\fX_n \cap (\R^d\times \log \VV)\big)
\cap \fK,\quad n\geq n_0.
$$
and, therefore,
$n\,\widetilde{\log}(\XX_n\cap (\R^d \times \VV))\cap \fK$ converges
in distribution to $\check{\fZ}_K\cap \fK$ as $n\to\infty$. In
particular, the above arguments show that the limit does not depend on
the choice of $\VV$.

Let us now explain the case of an arbitrary $\HH\subset \R^d\times\gl$
containing $(0,\II)$ and introduce assumptions that we shall impose on
$\HH$. Assume that the following objects exist:
\begin{itemize}
\item a neighbourhood $\mathfrak{U}\subset \R^d\times \log\VV$ of
  $(0,0)$ in $\RM$;
\item a neighbourhood $\mathbb{U}\subset \R^d\times \VV$ of $(0,\II)$
  in $\R^d\times \gl$;
\item a closed convex cone $\fC_{\HH}$ in $\RM$ with the apex at
  $(0,0)$;
\end{itemize}
such that
\begin{equation}\label{eq:H_condition}
  \HH\cap \mathbb{U}=\widetilde{\exp}(\fC_{\HH}\cap \mathfrak{U}).
\end{equation}
Informally speaking, condition \eqref{eq:H_condition} means that
locally around $(0,\II)$ the set $\HH$ is an image of a convex cone in
the tangent space $\RM$ under the extended exponential map
$\widetilde{\exp}$. The most important particular cases arise when
$\HH$ is the product of a linear space $\TT$ in $\R^d$ and a Lie
subgroup $\GG$ of $\gl$. In this situation $\fC_\HH$ is a linear
subspace of $\RM$ which is the direct sum of $\TT$ and the Lie algebra
$\fG$ of $\GG$. Furthermore, $\fC_\HH$ is a tangent space to $\HH$
(regarded as a product of smooth manifolds) at $(0,\II)$. In a more
general class of examples, we allow $\fC_\HH$ to be the direct sum of
$\TT$ and an arbitrary linear subspace of $\Matr$, which is not
necessarily a Lie algebra. In the latter case, the second component of
$\HH$ is not a Lie subgroup of $\gl$. %$\GG$.
Furthermore, $\fC_{\HH}$ can be
a proper cone, that is, not a linear subspace, so the second
component of $\HH$ is not necessarily a group. For example, assume
that $d=2$ and $\HH=\{0\}\times \GG$, where
\begin{displaymath}
  \GG=\left\{
      \begin{pmatrix}
        \lambda_1 & 0\\
        0 & \lambda_2
      \end{pmatrix},\lambda_1,\lambda_2\in (0,1]\right\}.
\end{displaymath}
Then \eqref{eq:H_condition} holds for appropriate $\mathbb{U}$ and
$\mathfrak{U}$ upon choosing 
\begin{displaymath}
  \fC_{\HH}=\{0\}\times \left\{
      \begin{pmatrix}
        \mu_1 & 0\\
        0 & \mu_2
      \end{pmatrix},\mu_1,\mu_2\leq 0\right\}.
\end{displaymath}
This example is important for the analysis of $(K,\HH)$-hulls because
we naturally want to exclude transformations that enlarge $K$. Examples
of a different kind, where $\HH$ is not a Lie subgroup, arise by
taking $\fC_{\HH}$ to be an arbitrary linear subspace of $\RM$ which
is not a Lie subalgebra.

For an arbitrary compact convex subset $\fK$ of $\RM$ which contains
the origin, the set $\fK\cap\fC_{\HH}$ is also compact convex and
contains the origin. Furthermore, there exists an $n_0\in\NN$ such
that
$$
\fK\cap \fC_{\HH} \subset n(\fC_{\HH}\cap \mathfrak{U})
=n\,\widetilde{\log}(\HH\cap \mathbb{U})
\subset n\,\widetilde{\log}(\HH\cap (\R^d\times \VV)),\quad n\geq n_0.
$$
Since $\fK\cap \fC_{\HH}$ is a compact convex set which contains the
origin,
\begin{align*}
  &\hspace{-1cm}n\,\widetilde{\log}\left(\XX_{K,\HH}(\Xi_n)
    \cap (\R^d\times \VV)\right)\cap(\fK\cap \fC_{\HH})\\
  &=n\,\widetilde{\log}\left(\XX_n
    \cap (\HH \cap (\R^d\times \VV))\right)\cap (\fK\cap \fC_{\HH} )\\
  &=n\fX_n \cap n\,\widetilde{\log}(\HH \cap (\R^d\times \VV))
    \cap (\fK\cap \fC_{\HH} )=n\fX_n \cap (\fK\cap \fC_{\HH})
\end{align*}
converges to $\check{\fZ}_K\cap \fK\cap \fC_{\HH}$ as $n\to\infty$. The limit
here is also independent of $\VV$.

Let us make a final remark in this informal discussion by connecting
the convergence of the sequence
$(n\,\widetilde{\log}\left(\XX_n\cap (\R^d\times \VV)\right))$ and
relation \eqref{eq:x_n_converges_to_k_minus_k}. The above argument
demonstrates that $\check{\fZ}_K$ necessarily contains a nonrandom set
\begin{align}
  \label{eq:rec_cone_of_fZ}
  R_K:&=\liminf_{n\to\infty} \left(n\,\widetilde{\log}
        ((K\ominus_{K,\R^d\times\gl}K)\cap (\R^d\times
        \VV))\right)\notag \\
      &=\bigcup_{k\geq 1}\bigcap_{n\geq k}\bigcap_{y\in K}
        \left\{(x,C)\in \RM: y\in \exp(C/n)(K+x/n)\right\},
\end{align}
which is unbounded. As we shall show, the set $R_{K}$ is, indeed,
contained in the recession cone of $\check{\fZ}_K$ which we identify in
Proposition~\ref{prop:rec_cone} below.

\section{Limit theorems for \texorpdfstring{$\XX_{K,\HH}(\Xi_n)$}{X{K,H}(Xi n)}}
\label{sec:poiss-point-proc}

Recall that $N(K,x)$ denotes the normal cone to $K$ at
$x\in\partial K$, where $K\in\sK_{(0)}^d$.  By $\Nor(K)$ we denote the
normal bundle, that is, a subset of $\partial K\times\Sphere$, which
is the family of $(x,N(K,x)\cap \Sphere)$ for $x\in\partial K$.  It is
known, see page~84 in \cite{schn2}, that $K$ has the unique outer unit
normal $u_K(x)$ at $x\in\partial K$ for almost all points $x$ with
respect to the $(d-1)$-dimensional Hausdorff measure
$\sH^{d-1}$. Denote the set of such points by $\Sigma_1(K)$,
so $\Sigma_1(K):=\{x\in\partial K: {\rm dim}\;N(K,x)=1\}$.

Let $\Theta_{d-1}(K,\cdot)$ be the generalised curvature measure of
$K$, see Section~4.2 in~\cite{schn2}.  The following formula, which is a
consequence of Theorem 3.2 in \cite{Hug:1998}, can serve as a
definition and is very convenient for practical purposes.  If $W$
is a Borel subset of $\R^{d}\times\Sphere$, then
$$
\Theta_{d-1}\big(K,(\partial K\times \Sphere )\cap W\big)
=\int_{\Sigma_1(K)} \one_{\{(x,u_K(x))\in W \}} \sH^{d-1}({\rm d}x).
$$
In particular, this formula implies that the support of
$\Theta_{d-1}(K,\cdot)$ is a subset of $\Nor(K)$ and its total mass is
equal to the surface area of $K$.

Let $\sP_K:=\sum_{i\geq 1}\delta_{(t_i,\eta_i,u_i)}$ be the Poisson
process on $(0,\,\infty)\times \Nor(K)$ with intensity measure $\mu$
being the product of Lebesgue measure on $(0,\infty)$ normalised by
$V_d(K)^{-1}$ and the measure $\Theta_{d-1}(K,\cdot)$.
If $K$ is strictly convex, $\sP_K$ can be
equivalently defined as a Poisson process
$\{(t_i,F(K,u_i),u_i), i\geq1\}$, where $\{(t_i,u_i),i\geq1\}$ is the
Poisson process on $(0,\infty)\times\Sphere$ with intensity being the
product of the Lebesgue measure on the half-line normalised by
$V_d(K)^{-1}$ and the surface area measure
$S_{d-1}(K,\cdot):=\Theta_{d-1}(K,\R^d\times\cdot)$ of $K$.

The notion of convergence of random closed sets in distribution with
respect to the Fell topology is recalled in the Appendix. For
$L\subset\RM$, denote by 
$$
\check{L}:=\big\{(-x,-C):(x,C)\in L\big\}
$$
the reflection of $L$ with respect to the origin in $\RM$.

\begin{theorem}\label{thm:main1}
  Assume that $K\in\sK^d_{(0)}$, and let $\fF$ be a closed convex set
  in $\RM$ which contains the origin. The sequence of random closed sets
  $((n\fX_n)\cap\fF)_{n\in\NN}$ converges in distribution in the space of
  closed subsets of $\RM$ endowed with the
  Fell topology to a random closed convex set $\check{\fZ}_K\cap\fF$, where
  \begin{equation}
    \label{thm:main1:claim}
    \fZ_{K}:=\bigcap_{(t,\eta,u)\in\sP_K}
    \Big\{(x,C)\in \RM:\langle C \eta+x,u\rangle \leq t\Big\}.
  \end{equation}
\end{theorem}

\begin{remark}
  While $\fX_n$ is not convex in general, the set
  $\check{\fZ}_K$ from \eqref{thm:main1:claim} is almost surely convex
  as an intersection of convex sets.
\end{remark}

Letting $\fF=\RM$ in Theorem~\ref{thm:main1} shows that $n\fX_n$
converges in distribution to $\check{\fZ}_K$.  If $\fF=\fK$ is a compact
convex set which contains the origin in $\RM$, the theorem
covers the setting of Section~\ref{sec:random}. Taking into account
the discussion there, we obtain the following.

\begin{corollary}
  Assume that $K\in\sK^d_{(0)}$, and let $\HH$ be a subset of
  $\R^d\times \gl$ which satisfies \eqref{eq:H_condition}. Then, the
  sequence of random closed sets
  \begin{displaymath}
%    \label{eq:xnH}
    n\,\widetilde{\log}\left(\XX_{K,\HH}(\Xi_n)\cap
      (\R^d\times \VV)\right)\cap \fC_{\HH},\quad n\in\NN,
  \end{displaymath}
  converges in distribution in the space of closed subsets of $\RM$
  endowed with the Fell topology to the random closed convex set
  $\check{\fZ}_{K}\cap \fC_{\HH}$.
\end{corollary}

The subsequent proof of Theorem~\ref{thm:main1} heavily relies on a
series of auxiliary results on the properties of the Fell topology and
convergence of random closed sets, which are collected in the
Appendix. We encourage the readers to acquaint themselves with the
Appendix before proceeding further.

We start the proof of Theorem~\ref{thm:main1} with an auxiliary result
which is an extension of Theorem~5.6 in \cite{MarMol:2021}. Let
$\sK^d_0$ be the space of compact convex sets in $\R^d$ containing the
origin and endowed with the Hausdorff metric. Let $L^{o}$ denote the
polar set to a closed convex set $L$ in $\R^d$, that is,
\begin{equation}\label{eq:polar}
  L^o:=\{x\in\R^d: h(L,x)\leq 1\}.
\end{equation}
In what follows we shall frequently use the relation
\begin{equation}\label{eq:half-space_polar}
  [0,t^{-1}u]^o=H^{-}_{u}(t),\quad t>0,\quad u\in\Sphere,
\end{equation}
where 
$$
H^{-}_{u}(t):=\{x\in\R^d:\langle x,u\rangle\leq t\},\quad t\in\R,\quad u\in\Sphere.
$$

From Theorem~5.6 in \cite{MarMol:2021} we know that
\begin{displaymath}
  \sum_{k=1}^{n}\delta_{n^{-1}(K-\xi_k)^{o}}\;\dodn\;
  \sum_{(t,\eta,u)\in\sP_K}\delta_{[0,\,t^{-1}u]}
  \quad \text{as}\quad n\to\infty,
\end{displaymath}
where the convergence is understood as the convergence in distribution
on the space of point measures on\footnote{It would be more precise to
  write $\sK^d_0\setminus\{\{0\}\}$ instead of
  $\sK^d_0\setminus\{0\}$, but we prefer the latter notation for the
  sake of notational simplicity.} $\sK^d_0\setminus\{0\}$ endowed with
the vague topology. The limiting point process consists of random
segments $[0,\,x]$ with $x=t^{-1}u$ derived from the first and third
coordinates of $\sP_K$. Regarding $\xi_k$ as a mark of
$n^{-1}(K-\xi_k)^{o}$ for $k=1,\ldots,n$, we have the following
convergence of marked point processes, which strengthens the above
mentioned result from \cite{MarMol:2021}.

\begin{lemma}\label{lem:marked_point_processes}
  Assume that $K\in\sK_{(0)}^d$. Then
  \begin{equation}
    \label{eq:17}
    \sum_{k=1}^{n}\delta_{(n^{-1}(K-\xi_k)^{o},\,\xi_k)}
    \;\dodn\;
    \sum_{(t,\eta,u)\in \sP_K}\delta_{([0,t^{-1}u],\eta)}
    \quad \text{as}\quad n\to\infty,
  \end{equation}
  where the convergence is understood as the convergence in
  distribution on the space of point measures on
  $(\sK^d_0\setminus\{0\})\times \R^d$ endowed with the vague topology.
\end{lemma}
\begin{proof}
  Let $p(\partial K,\cdot)$ be the metric projection on $K$, that is,
  $p(\partial K,x)$ is the set of closest to $x$ points on
  $\partial K$. We start by noting that for the limiting Poisson
  process the following equality holds for all
  $L\in\sK^d_0\setminus\{0\}$ and every Borel $R\subseteq\R^d$
  \begin{align*}
    \P\big\{[0,t^{-1}u]\subset L \;\text{or}\;
    & \eta\notin p(\partial K,R)\text{ for all } (t,\eta,u)\in\sP_K\big\}\\
    &=\exp\big(-\mu(\{(t,\eta,u):
      [0,t^{-1}u]\not\subset L, \eta\in p(\partial K,R)\}\big)\\
    &=\exp\big(-\mu(\{(t,\eta,u):
      L^{o}\not\subset H^{-}_{u}(t), \eta\in p(\partial K,R)\}\big)\\
    &=\exp\big(-\mu(\{(t,\eta,u):
      h(L^o,u)>t, \eta\in p(\partial K,R)\}\big)\\
    &=\exp\Big(-\frac{1}{V_d(K)}
      \int_{\Nor(K)} \one_{\{x\in p(\partial K,R)\}}h(L^o,u)
      \Theta_{d-1}(K,{\rm d}x\times {\rm d}u)\Big).
  \end{align*}

  According to Proposition~\ref{prop:marked-sets} in the Appendix,
  see, in particular, Eq.~\eqref{eq:10v}, we need to show that, for
  every $L\in\sK^d_0\setminus\{0\}$ and every Borel $R\subseteq\R^d$,
  as $n\to\infty$,
  \begin{equation}\label{eq:kid_rataj1}
    n\P\{n^{-1}(K-\xi)^{o}\not\subseteq L,\xi\in R\}
    \;\longrightarrow\;
    \frac{1}{V_d(K)}\int_{\Nor(K)}\one_{\{x\in p(\partial K,R)\}}
      h(L^{o},u)\Theta_{d-1}(K,{\rm d}x\times{\rm d}u).
  \end{equation}
  Note that
  \begin{align*}
    \P\{n^{-1}(K-\xi)^{o} &\not\subseteq L,\xi\in R\}
    =\P\{n^{-1}L^{o}\not\subseteq (K-\xi),\xi\in R\}\\
    &=\P\{\xi\not\in K\ominus n^{-1}L^{o},\xi\in R\}
    =\frac{V_d(R\cap (K\setminus(K\ominus n^{-1}L^{o})))}{V_d(K)}.
  \end{align*}
  Applying Theorem~1 in \cite{kid:rat06} with $C=p(\partial K,R)$,
  $A=K$, $P=B=W=\{0\}$, $Q=-(L^{o})$ and $\eps=n^{-1}$, we obtain
  \eqref{eq:kid_rataj1}. The proof is complete.
\end{proof}

Applying continuous mapping theorem to convergence~\eqref{eq:17}
and using Lemma~\ref{lemma:polar-map}(ii), we obtain the convergence of
marked point processes 
\begin{equation}
  \label{eq:marked_point_processes_conv2}
  \sum_{k=1}^{n}\delta_{(n(K-\xi_k),\,\xi_k)}\;
  \dodn\;\sum_{(t,\eta,u)\in\sP_K}\delta_{(H^{-}_{u}(t),\,\eta)}
  \quad \text{as}\quad n\to\infty.
\end{equation}

\begin{proof}[Proof of Theorem \ref{thm:main1}]
  According to Lemma~\ref{lemma:restriction} in the Appendix it 
  suffices to show that $(n\fX_n)\cap\fF\cap\fLoc$ converges to
  $\check{\fZ}_K\cap \fF\cap \fLoc$ for an arbitrary compact convex
  subset $\fLoc$ of $\RM$, which contains the origin in its interior
  and then pass to the limit $\fLoc\uparrow (\RM)$. It holds
  \begin{align*}
    (n\fX_n)\cap\fLoc
    &=\Big\{(x,C)\in\fLoc:\Xi_n\subseteq \exp(C/n)(K+x/n) \Big\}\\
    &=\bigcap_{k=1}^{n}\Big\{(x,C)\in\fLoc:\xi_k\in
      \exp(C/n)(K+x/n) \Big\}\\
    &=\bigcap_{k=1}^{n}\Big\{(x,C)\in\fLoc:\exp(-C/n)
      \xi_k\in K+x/n \Big\}\\
    &=\bigcap_{k=1}^{n}\Big\{(x,C)\in\fLoc:\big(n(\exp(-C/n)-
      \II)\big)\xi_k-x \in n(K-\xi_k)\Big\}.
  \end{align*}
  Let
  \begin{displaymath}
    a_m:=\sup_{n\geq m}\sup_{(x,C)\in\fLoc,y\in K} \big\|\big(n(\exp(-C/n)-
    \II\big)y+Cy \big\|,\quad m\in\NN.
  \end{displaymath}
  Note that $a_m\to0$ as $m\to\infty$, because 
  $n\big(\exp(-C/n)-\II\big)\to -C$ locally uniformly in $C$ as $n\to\infty$,
  the set $\fLoc$ is compact in $\RM$, and $K$ is compact in
  $\R^d$.

  Let $B_{a_m}$ be the closed ball of radius $a_m$ in $\R^d$
  centred at the origin.  For each $m\in\NN$ and $n\geq m$, we have
  \begin{equation}
    \label{eq:5}
    \fY_{m,n}^-\subset \big((n\fX_n)\cap\fLoc\big) \subset\fY_{m,n}^+,
  \end{equation}
  where
  \begin{displaymath}
    \fY_{m,n}^+:=\bigcap_{k=1}^{n}\big\{(x,C)\in\fLoc:
    -C\xi_k-x \in n(K-\xi_k)+B_{a_m}\big\}
  \end{displaymath}
  and
  \begin{equation}
    \label{eq:minus}
    \fY_{m,n}^-:=\bigcap_{k=1}^{n}\big\{(x,C)\in\fLoc:
    -C\xi_k-x +B_{a_m} \subset n(K-\xi_k)\big\}.
  \end{equation}
  The advantage of lower and upper bounds in \eqref{eq:5} is the
  convexity of $\fY_{m,n}^{\pm}$, which makes their analysis
  simpler. We aim to apply Lemma~\ref{lemma:approx} from the Appendix
  with $Y_{m,n}^{\pm}=\fY_{m,n}^{\pm}\cap \fF$ and
  $X_n=(n\fX_n)\cap \fLoc\cap \fF$.
  % Let us start with analysis of $\fY_{m,n}^{+}$ which is simpler.
  
 Let $\fL$ be a compact subset of $\fLoc$.  Denote
  \begin{align*}
    M_m^+(\fL)&:=\big\{(L,y)\in \sK_0^d\times\R^d:
                -C y-x \in L^o+B_{a_m}\;\text{for all}\; (x,C)\in\fL\big\},\\
    M_m^-(\fL)&:=\big\{(L,y)\in \sK_0^d\times\R^d: -C y-x + B_{a_m}\subset L^o
                \;\text{for all}\; (x,C)\in\fL\big\},\\
    M(\fL)&:=\big\{(L,y)\in \sK_0^d\times\R^d: -C y-x \in L^o
            \;\text{for all}\; (x,C)\in\fL\big\}.
  \end{align*}
  Then
  \begin{displaymath}
    \Prob{\fL\subset \fY_{m,n}^{\pm}}
    =\Prob{(n^{-1}(K-\xi_i)^o,\xi_i)\in M_m^{\pm}(\fL),i=1,\dots,n}.
  \end{displaymath}

  By Lemma~\ref{lem:marked_point_processes}, the point process
  $\{(n^{-1}(K-\xi_i)^o,\xi_i),i=1,\dots,n\}$ converges in
  distribution to the Poisson process
  $\big\{([0,t^{-1}u],\eta):(t,\eta,u)\in\sP_K\big\}$.  The sets $M(\fL)$,
  $M_m^{\pm}(\fL)$ are continuity sets for the distribution of the
  limiting Poisson process. Indeed, for each
  $(t,\eta,u)\in(0,\infty)\times\Nor(K)$,
  \begin{align*}
    &\Big\{([0,t^{-1}u],\eta)\in \partial M_m^+(\fL)\Big\}\\
    &=\Big\{-C \eta-x\in H^{-}_{u}(t+a_m)\;
      \text{for all}\; (x,C)\in\fL \Big\}
      \setminus \Big\{-C \eta-x\in \Int H^{-}_{u}(t+a_m)\;\text{for
      all}\; (x,C)\in\fL \Big\}\\
    &=\Big\{\langle -C\eta-x,u\rangle \leq t+a_m
      \;\text{for all}\; (x,C)\in\fL \Big\}
      \setminus \Big\{\langle -C\eta-x,u\rangle< t+a_m
      \;\text{for all}\; (x,C)\in\fL \Big\}\\
    &=\Big\{\langle -C\eta-x,u\rangle \leq t+a_m
      \;\text{for all}\; (x,C)\in\fL\;
      \text{and}\;\langle -C\eta-x,u\rangle = t+a_m
      \;\text{for some}\; (x,C)\in\fL \Big\},
  \end{align*}
  where $\Int$ denotes the topological interior.  Since the
  probability of the latter event for some $(t,\eta,u)\in\sP_K$
  vanishes, it follows that $M_m^{+}(\fL)$ is a continuity set for the
  Poisson point process
  $\{([0,t^{-1}u],\eta):(t,\eta,u)\in\sP_K\}$. Letting $a_m=0$, we
  obtain that $M(\fL)$ is also a continuity set. The argument for
  $M_m^-(\fL)$ is similar by replacing $a_m$ with $(-a_m)$.

  Thus, for all $m\in\NN$,
  \begin{displaymath}
    \Prob{\fL\subset \fY_{m,n}^+}
    \to \Prob{\big\{([0,t^{-1}u],\eta):(t,\eta,u)\in\sP_K\big\}\subset M_m^+(\fL)}
    =\Prob{\fL\subset \fY_m^{+}}
  \end{displaymath}
  as $n\to\infty$, where
  \begin{displaymath}
    \fY_m^+:=\bigcap_{(t,\eta,u)\in\sP_K}
    \Big\{(x,C)\in\fLoc:-C\eta-x \in H^-_u(t+a_m) \Big\}.
  \end{displaymath}
  The random closed sets $\fY_{m,n}^+$ and $\fY_m^+$ are convex and
  almost surely contain a neighbourhood of the origin in $\RM$, hence,
  are regular closed, see the Appendix for the definition. By
  Theorem~\ref{thm:inclusion} applied to the space $\RM$, the random
  convex set $\fY_{m,n}^+$ converges in distribution to
  $\fY_m^+$. Since $\fY_{m,n}^+\dodn \fY_m^+$ as $n\to\infty$, and
  the involved sets almost surely contain the origin in their
  interiors, Corollary~\ref{cor:intersection} yields that
  $(\fY_{m,n}^+\cap\fF)\dodn (\fY_m^+\cap\fF)$ as $n\to\infty$, for
  each closed convex set $\fF$ which contains the origin in $\RM$ (not
  necessarily as an interior point). Thus, we have checked part (i) of
  Lemma~\ref{lemma:approx}.

  We proceed with checking part (ii) of Lemma~\ref{lemma:approx} with
  $Y_m^{-}:=\fY_m^-$, where
  \begin{displaymath}
    \fY_m^-:=\bigcap_{(t,\eta,u)\in\sP_K}
    \Big\{(x,C)\in\fLoc:-C\eta-x \in  H^-_u(t-a_m) \Big\}.
  \end{displaymath} 
  Note that the random sets $\fY_{m,n}^-$ and $\fY_m^-$ may be empty
  and otherwise not necessarily contain the origin.  We need to
  check~\eqref{eq:11}, which in our case reads as follows
  \begin{equation}\label{eq:main_theorem_proof1}
    \Prob{\fY_{m,n}^-\cap\fF\cap\fL\neq\emptyset,0\in \fY_{m,n}^-}
    \to \Prob{\fY_{m}^-\cap\fF\cap\fL\neq\emptyset,0\in \fY_{m}^-}
    \quad\text{as}\quad n\to\infty,
  \end{equation}
  for all compact sets $\fL$ which are continuity sets of
  $\fY_{m}^-\cap\fF$. We shall prove~\eqref{eq:main_theorem_proof1}
  for all compact sets $\fL$ in $\RM$.  To this end, we shall employ
  Lemma~\ref{lemma:intersection} and divide the
  derivation~\eqref{eq:main_theorem_proof1} into several steps, each
  devoted to checking one condition of Lemma~\ref{lemma:intersection}.
  
  \noindent
  {\sc Step 1.} Let us check that, for sufficiently large $n\in\NN$,
  \begin{displaymath}
    \Prob{(0,0)\in \fY_{m,n}^-}=\Prob{(0,0)\in \Int\fY_{m,n}^-}>0,\quad m\in\NN.
  \end{displaymath}
  Since the interior of a finite intersection is the intersection of the
  interiors, and using independence, it suffices to check this for
  each of the sets which appear in the intersection
  in~\eqref{eq:minus}. If $(0,0)$ belongs to
  $Y_k:=\big\{(x,C)\in\fLoc: -C\xi_k-x +B_{a_m} \subset
  n(K-\xi_k)\big\}$, then $B_{a_m} \subset n(K-\xi_k)$. Since $\xi_k$
  is uniform on $K$, we have
  \begin{displaymath}
    \Prob{B_{a_m} \subset n(K-\xi_k)}=\Prob{B_{a_m} \subset n\Int(K-\xi_k)}
  \end{displaymath}
  for all $n$. If $B_{a_m} \subset n\Int(K-\xi_k)$,
  then $-C\xi_k-x +B_{a_m} \subset n(K-\xi_k)$ for all $x$ and $C$
  from a sufficiently small neighbourhood of the origin in
  $\RM$. Furthermore, 
  \begin{displaymath}
    \Prob{(0,0)\in \fY_{m,n}^-}=\Prob{B_{a_m} \subset
      n(K-\xi_k),k=1,\dots,n}
    =\Prob{B_{a_m/n}\subset K\ominus\Xi_n}>0
  \end{displaymath}
  for all sufficiently large $n$. 
  
  \noindent
  {\sc Step 2.}  Let us check that, for each $m\in\NN$,
  \begin{displaymath}
    \Prob{(0,0)\in \fY_{m}^-}=\Prob{(0,0)\in \Int\fY_{m}^-}>0.
  \end{displaymath}
  The equality above follows from the observation that the origin lies
  on the boundary of
  $\big\{(x,C)\in\fLoc:-C\eta-x \in H^-_u(t-a_m) \big\}$ only if $t=a_m$,
  which happens with probability zero. Furthermore,
  $(0,0)\in \fY_{m}^-$ if $t\geq a_m$ for all $(t,\eta,u)\in\sP_K$,
  which has positive probability.
  
  {\sc Step 3.} By a similar argument as we have used for
  $\fY_{m,n}^+$, for every compact subset $\fL$ of $\fLoc$ and
  $m\in\NN$, we have
  $$
    \Prob{\fL\subset \fY_{m,n}^-}
    \to \Prob{\big\{([0,t^{-1}u],\eta):(t,\eta,u)\in\sP_K\big\}
      \subset M_m^-(\fL)}=\Prob{\fL\subset \fY_m^{-}}
    \quad \text{as }\quad n\to\infty.
  $$
  
  Summarising, we have checked all conditions of
  Lemma~\ref{lemma:intersection}. This finishes the proof
  of~\eqref{eq:main_theorem_proof1} and shows that all conditions
  of part (ii) of Lemma~\ref{lemma:approx} hold. It remains to note
  that
  \begin{equation*}
    (\fY_m^+\cap\fF)\downarrow (\check{\fZ}_K\cap\fLoc\cap\fF),
    \quad(\fY_{m}^-\cap\fF)\uparrow(\check{\fZ}_K\cap\fLoc\cap\fF)
    \quad \text{a.s. as}\quad m\to\infty
  \end{equation*}
  in the Fell topology, and
  $$
  \lim_{m\to\infty}\Prob{0\in \fY_{m}^-}=1.
  $$ 
  Thus, by Lemma~\ref{lemma:approx}  
  $(n\fX_n)\cap\fLoc\cap\fF$ converges in distribution to
  $\check{\fZ}_K\cap\fLoc\cap\fF$ as $n\to\infty$. By
  Lemma~\ref{lemma:restriction}, $(n\fX_n)\cap\fF$ converges in
  distribution to $\check{\fZ}_K\cap\fF$.
\end{proof}

\section{Properties of the set \texorpdfstring{$\fZ_K$}{ZK}}
\label{sec:zero-cells-matrix}

\subsection{Boundedness and the recession cone}

The random set $\fZ_{K}$ is a subset of the product space
$\RM$. The latter space can be turned into the
real Euclidean vector space with the inner product given by
\begin{displaymath}
  \langle (x,C_1), (y,C_2) \rangle_1
  :=\langle x,y\rangle +\Tr (C_1C_2^{\top}),
  \quad x,y\in\R^d,\quad C_1,C_2\in\Matr,
\end{displaymath}
where $\Tr$ denotes the trace of a square matrix and $C^{\top}$ is the
transpose of $C\in\Matr$. In terms of this inner product the set
$\fZ_K$ can be written as
\begin{equation}\label{eq:zero_cell_general}
  \fZ_{K}=\bigcap_{(t,\eta,u)\in\sP_K}\Big\{(x,C)\in\RM
  :\langle (x,C),(u,\eta\otimes u)\rangle_1\leq t\Big\}
  =\bigcap_{(t,\eta,u)\in\sP_K}H^{-}_{(u,\eta\otimes u)}(t),
\end{equation}
where $H^{-}_{(u,\eta\otimes u)}(t)$ is a closed half-space of $\RM$
containing the origin, and $\eta\otimes u$ is the tensor product of
$\eta$ and $u$. The boundaries of $H^{-}_{(u,\eta\otimes u)}(t)$,
$(t,\eta,u)\in\sP_K$, constitute a Poisson process on the affine
Grassmannian of hyperplanes in $\RM$ called a Poisson hyperplane
tessellation. The random set obtained as the intersection of the
half-spaces $H^{-}_{(u,\eta\otimes u)}(t)$, $(t,\eta,u)\in\sP_K$, is
called the zero cell, see Section~10.3 in~\cite{sch:weil08}. The
intensity measure of this tessellation is the measure on the affine
Grassmannian obtained as the product of the Lebesgue measure on $\R_+$
(normalised by $V_d(K)$) and the measure $\nu_K$ obtained as the
pushforward of the generalised curvature measure
$\Theta_{d-1}(K,\cdot)$ under the map
$\Nor(K)\ni (x,u)\mapsto (u,x\otimes u)\in \RM$.  If, for example,
$K=B_1$ is the unit ball, $\nu_K$ is the pushforward of the
$(d-1)$-dimensional Hausdorff measure on the unit sphere $\Sphere$ by
the map $u\mapsto (u,u\otimes u)$.  For a strictly convex and smooth
body $K$, the positive cone generated by
$\{x\otimes u: (x,u)\in \Nor(K)\}$ is called the normal bundle cone of
$K$, see \cite{gruber2014normal}.

Note that the set $\fZ_K$ is almost surely convex, closed and unbounded. Thus,
it is natural to consider the recession cone of $\fZ_K$ which is formally
defined as
$$
{\rm rec}(\fZ_K):=\big\{(x,C)\in \RM: (x,C)+\fZ_K\subseteq \fZ_K\big\}.
$$
For instance, since $\fZ_K$ always contains $(0,-\lambda\II)$ with
$\lambda\geq 0$, the recession cone also contains
$\{(0,-\lambda\II):\lambda\geq0\}$. Recall that $S(K,y)$ is the
supporting cone to $K$ at $y$, see \eqref{eq:supporting-cone}.

\begin{proposition}\label{prop:rec_cone}
  The set $R_K$ defined at \eqref{eq:rec_cone_of_fZ} is contained in
  the following set
  \begin{equation}\label{eq:rec_cone_of_fZ_alt}
    T_K:=\bigcap_{y\in K}\left\{(x,C)\in \RM: -Cy-x\in S(K,y)\right\},
  \end{equation}
  which is a closed convex cone in $\RM$. Furthermore, with probability one
  \begin{equation}\label{eq:rec_cone_claim_1}
    \check{T}_K\subset {\rm rec}(\fZ_K).
  \end{equation}
  Moreover, if $K$ is smooth, then
%  \begin{displaymath}
    % \label{eq:rec_cone_claim_11}
    $\check{T}_K={\rm rec}(\fZ_K)$, 
%  \end{displaymath}
  and, with $\HH$ satisfying \eqref{eq:H_condition},
  \begin{equation}\label{eq:rec_cone_claim_2}
    {\rm rec}(\fZ_K\cap\fC_{\HH})=\check{T}_K\cap\fC_{\HH}.
  \end{equation}
  In particular, the limit $\check{\fZ}_K\cap\fC_{\HH}$ of
  $(n\,\widetilde{\log}\left(\XX_{K,\HH}(\Xi_n)\cap (\R^d\times
    \VV)\right))\cap\fC_{\HH}$ is a random compact set with probability one if and
  only if $\check{T}_K\cap\fC_{\HH}=\{(0,0)\}$.
\end{proposition}
\begin{proof}
  It is clear that $R_K\subseteq \cap_{y\in K} R_{K,y}$, where 
  \begin{displaymath}
    R_{K,y}:=\bigcup_{k\geq 1}\bigcap_{n\geq k}
    \left\{(x,C)\in \RM: y\in \exp(C/n)(K+x/n)\right\}.
  \end{displaymath}
  A pair $(x,C)\in\RM$ lies in $R_{K,y}$ if and only if there exists a
  $k\in\NN$ such that $\exp(-C/n)y-x/n\in K$ for all $n\geq k$,
  equivalently, 
  \begin{displaymath}
    n(\exp(-C/n)-\II)y-x\in n(K-y)\quad\text{for all}\quad n\geq k.
  \end{displaymath}
  Letting $n\to\infty$ on the left-hand side and using that
  $\limsup_{n\to\infty}n(K-y)=S(K,y)$ show that $-Cy-x\in S(K,y)$.
%  which means that $(x,C)\in T_{K,y}$.
  Thus,
  \begin{displaymath}
    R_{K,y}\subset \left\{(x,C)\in \RM: -Cy-x\in S(K,y)\right\}=:T_{K,y},
  \end{displaymath}
  so that $R_K\subseteq T_K$. Since $T_{K,y}$ is a closed convex cone
  for all $y\in K$, the set $T_K$ is a closed convex cone as well.

  In order to check \eqref{eq:rec_cone_claim_1} note that
  $(N(K,y))^{o}=S(K,y)$. Hence, $-Cy-x\in S(K,y)$ if and only if
  $\langle Cy+x,u\rangle\geq 0$ for all $u\in N(K,y)$. Therefore,
  \begin{align*}
    T_K&=\bigcap_{y\in K}\bigcap_{u\in N(K,y)}
         \left\{(x,C)\in \RM: \langle Cy+x,u\rangle\geq 0\right\}\\
       &=\bigcap_{(y,u)\in \Nor(K)}\left\{(x,C)\in \RM:
         \langle Cy+x,u\rangle\geq 0\right\},
  \end{align*}
  where we have used that $N(K,y)=\{0\}$ if $y\in \Int K$.

  It follows from well-known results on recession cones, see page~62 in
  \cite{Rockafellar:1972}, that
  \begin{align*}
    {\rm rec}(\fZ_{K})
    &=\bigcap_{(t,\eta,u)\in\sP_K}\Big\{(x,C)\in\RM:
      \langle (x,C),(u,\eta\otimes u)\rangle_1\leq 0\Big\}\\
    &=\bigcap_{(t,\eta,u)\in\sP_K}\Big\{(x,C)\in\RM:
      \langle C\eta+x,u\rangle\leq 0\Big\}.
  \end{align*}
  This immediately yields that $\check{T}_K\subseteq {\rm rec}(\fZ_{K})$. To see
  the converse inclusion for smooth $K$ note that the set
  \begin{displaymath}
    \{(\eta,u)\in \Nor(K):(t,\eta,u)\in \sP_K\text{ for some }t>0\}
  \end{displaymath}
  is a.s.~dense in $\Nor(K)=\{(x,u_K(x)):x\in\partial K\}$, where
  $u_K(x)$ is the unique unit outer normal to $K$ at $x$, see Lemma~4.2.2
  and Theorem~4.5.1 in \cite{schn2}. Thus, with probability one, for
  every $(x,C)\in {\rm rec}(\fZ_{K})$ and $(y,u)\in \Nor(K)$ there
  exists a sequence $(\eta_n,u_n)$ such that $(\eta_n,u_n) \to (y,u)$
  as $n\to\infty$, and $\langle C\eta_n+x,u_n\rangle\leq 0$ for all $n$. Thus,
  $\langle Cy+x,u\rangle\leq 0$ and $(x,C)\in \check{T}_K$.  Finally,
  relation \eqref{eq:rec_cone_claim_2} follows from Corollary~8.3.3 in
  \cite{Rockafellar:1972} since
  \begin{displaymath}
    {\rm rec}(\fZ_k\cap \fC_{\HH})
    ={\rm rec}(\fZ_k)\cap {\rm rec}(\fC_{\HH})
    ={\rm rec}(\fZ_k)\cap \fC_{\HH}. \qedhere
  \end{displaymath}
\end{proof}

Further information on the properties of $\fZ_K$ is encoded in its
polar set which takes the following rather simple form
\begin{displaymath}
  \fZ_K^o=\conv \Big(\bigcup_{(t,\eta,u)\in\sP_K} [0,t^{-1}(u,(\eta\otimes u))]\Big),
\end{displaymath}
which easily follows from \eqref{eq:half-space_polar}. Since $\fZ_K$
a.s.~contains the origin in the interior, $\fZ_K^o$ is
a.s. compact. Note that $\fZ_K^o$ is a subset of the Cartesian product
of $\R^d$ and Gruber's normal bundle cone, see
\cite{gruber2014normal}. The projection of $\fZ_K^o\cap(\R^d\times\{0\})$ on the first
factor $\R^d$ is a random polytope with probability one, which was
recently studied in \cite{MarMol:2021}, see Section~5.1 therein.

\subsection{Affine transformations of \texorpdfstring{$K$}{K}}

Let us now derive various properties of $\fZ_K$ with respect to
transformations of $K$. It is easy to see that
$\fZ_{rK}$ coincides in distribution with $r^{-1}\fZ_K$, for every
fixed $r>0$.  Let $A\in \ortho$ be a fixed orthogonal matrix. Note
that the point process $\sP_{AK}$ has the same distribution as the
image of $\sP_K$ under the map $(t,\eta,u)\mapsto(t,A\eta,Au)$. Then,
with $\od$ denoting equality of distributions,
\begin{align*}
  \fZ_{AK}&\od \bigcap_{(t,\eta,u)\in\sP_K}
    \Big\{(x,C)\in\RM:\langle C A\eta+x,Au\rangle \leq
    t\Big\}\\
   &=\bigcap_{(t,\eta,u)\in\sP_K}
    \Big\{(x,C)\in\RM:\langle A^\top C A\eta+ A^\top x,u\rangle \leq
    t\Big\}\\
    &=\bigcap_{(t,\eta,u)\in\sP_K}
    \Big\{(Ay,ABA^{\top})\in\RM:\langle B\eta+ y,u\rangle \leq
    t\Big\},
\end{align*}
so that $\fZ_{AK}$ has the same distribution as $\fZ_K$ transformed
using the map $\mathcal{O}_A:(x,C)\to (Ax,ACA^{\top})$, which is
orthogonal with respect to the inner product
$\langle\cdot,\cdot\rangle_1$ in $\RM$.  In particular, if $K$ is
invariant under $A$, then the distribution of $\fZ_K$ is invariant
under $\mathcal{O}_A$. Most importantly, if $K$ is a Euclidean ball,
then the distribution of $\fZ_K$ is invariant under $\mathcal{O}_A$
for any $A\in\ortho$.
%that is, $\fZ_K$ may be called isotropic.

If $K$ is translated by $v\in\R^d$, then
\begin{align*}
  \fZ_{K+v}&\od \bigcap_{(t,\eta,u)\in\sP_K}
  \Big\{(x,C)\in\RM:\langle C\eta,u\rangle
  +\langle Cv,u\rangle+\langle x,u\rangle\leq t\Big\}\\
  &=\bigcap_{(t,\eta,u)\in\sP_K}
  \Big\{(x-Cv,C)\in\RM:\langle C\eta,u\rangle
  +\langle x,u\rangle\leq t\Big\},
\end{align*}
meaning that $\fZ_{K+v}$ is the image of $\fZ_K$ under the linear
operator $(x,C)\mapsto (x-Cv,C)$.

\section{Examples}\label{sec:examples}

Throughout this section $\Xi_n$ is a sample from the uniform
distribution on $K$. If $\GG$ consists of the unit matrix and
$\TT=\R^d$, then Theorem~\ref{thm:main1} turns into Theorem~5.1 of
\cite{MarMol:2021}. Another object which has been recently treated in
the literature is given in Example~\ref{ex:cones} below.

Consider further examples involving nontrivial matrix groups.

\begin{example}[General linear group]
  Let $\GG=\gl$ be the general linear group, so that $\fG$ is the family
  $\Matr$. If $\TT=\R^d$, Proposition~\ref{prop:largest_group} shows that
  $Q_n=\conv(\Xi_n)$ for every choice of $K\in \sK^d_{(0)}$.

  Assume that $\TT=\{0\}$, and let $K$ be the unit ball $B_1$. Then
  $Q_n$ is strictly larger than $\conv(\Xi_n)$ with probability
  $1$. Indeed, it is clear that $Q_n\supseteq \conv(\Xi_n)$, and the
  inclusion is strict because the set $Q_n$ is symmetric with respect
  to the origin, while the set $\conv(\Xi_n)$ is almost surely
  not. Then $\fC_{\HH}=\{0\}\times\Matr$ and
  \begin{displaymath}
    \fZ_{B_1}\cap (\{0\}\times\Matr) =\{0\}\times \bigcap_{(t,u)\in\sP}
    \big\{C\in\Matr :  \langle Cu,u\rangle\leq t\big\}, 
  \end{displaymath}
  where $\sP$ is the Poisson process on $\R_+\times\Sphere$ with
  intensity being the product of the Lebesgue measure multiplied by
  $d$ and the uniform probability measure on $\Sphere$. The factor $d$
  results from taking the ratio of the surface area of the unit sphere
  and the volume of the unit ball.
  
  By Proposition \ref{prop:rec_cone}, since
  $S(B_1,y)=H^{-}_y(0)=\{x\in\R^d: \langle x,y\rangle\leq 0\}$
  \begin{displaymath}
    {\rm rec}(\fZ_{B_1}\cap (\{0\}\times\Matr))
    =\{0\}\times \bigcap_{y\in B_1}\big\{C\in\Matr : \langle Cy,y\rangle
    \leq 0 \big\}.
  \end{displaymath}
  Thus,
  \begin{displaymath}
    {\rm rec}\big (\check{\fZ}_{B_1}\cap (\{0\}\times\Matr) \big)
    =\{0\}\times \{C\in\Matr : C\text{ is positive semi-definite}\}.
  \end{displaymath}
  In particular, the second factor contains the subspace $\Matr[SSym]$
  of all skew-symmetric matrices, as well as all real symmetric
  matrices with nonnegative eigenvalues. The former reflects the fact
  that $B_1$ is invariant with respect to $\ortho$ for which
  $\Matr[SSym]$ is the Lie algebra. The latter is a consequence of the
  fact that $\XX_{B_1,\gl}(\Xi_n)$ contains all scalings with scaling
  factors (possibly different along pairwise orthogonal directions)
  larger than or equal to $1$.
\end{example}

\begin{example}[Special linear group]
  Let $\GG$ be the special linear group $\mathbb{SL}_d$, which
  consists of all $d\times d$ real-valued matrices with determinant
  one, and assume again that $\TT=\{0\}$. The elements of the
  corresponding Lie algebra $\fG= \{C\in \Matr: \Tr C = 0\}$ are
  matrices with zero trace. Thus, we can set
  $\fC_{\HH}=\{0\}\times\fG$. If $K=B_1$, then
  \begin{displaymath}
    \fZ_{B_1}\cap (\{0\}\times\fG)=\{0\}\times \bigcap_{(t,u)\in\sP}
    \big\{C\in \fG:\langle C u,u\rangle \leq t\big\}.
  \end{displaymath}
  By Proposition~\ref{prop:rec_cone}
  \begin{align*}
    \mathfrak{R}:={\rm rec}\big (\fZ_{B_1}\cap (\{0\}\times\fG) \big)
    &=\{0\}\times \bigcap_{y\in B_1}\big\{C\in \fG:\langle Cy,y\rangle\leq
      0 \big\}\\
    &=\{0\}\times \bigcap_{y\in B_1}\big\{C\in \Matr:
      \Tr C=0, \langle Cy,y\rangle\leq 0 \big\}.
  \end{align*}
  The intersection of $\mathfrak{R}$ and $\check{\mathfrak{R}}$ is
  called the lineality space of $\fZ_{B_1}\cap (\{0\}\times\fG)$; it
  consists of all vectors that are parallel to a line contained in
  $\fZ_{B_1}\cap (\{0\}\times\fG)$, see page~16 in~\cite{schn2}. Clearly,
  the lineality space of $\fZ_{B_1}\cap (\{0\}\times\fG)$ is
  a.s. equal to
  \begin{displaymath}
    \mathfrak{R}\cap \check{\mathfrak{R}}
    =\{0\}\times \bigcap_{y\in B_1}\big\{C\in \Matr: \Tr C=0,
    \langle Cy,y\rangle=0\big\}=\{0\}\times \Matr[SSym].
  \end{displaymath}
  The vector space of square matrices $\Matr$ is the direct sum of the
  vector spaces of symmetric and skew-symmetric matrices:
  \begin{displaymath}
    \Matr=\Matr[Sym] \oplus \Matr[SSym].
  \end{displaymath}
  Furthermore, with respect to the inner product
  $\langle A, B \rangle := \Tr (A B^{\top})$ this direct sum
  decomposition is orthogonal. Similarly, the space $\fG$ is a direct
  sum of two vector spaces $\Matr[SSym]$ and $\fG_+$, where
  $\fG_+ := \{C\in \Matr[Sym]: \Tr C = 0\}$. By Lemma~1.4.2 in
  \cite{schn2} we a.s.~have the orthogonal decomposition
  \begin{displaymath}
    \fZ_{B_1}\cap (\{0\}\times\fG)
    =\{0\}\times\left(\Matr[SSym] \oplus
      \big(\fZ_{B_1}\cap (\{0\}\times\fG)\big)_{+}\right),
  \end{displaymath}
  where 
  \begin{align*}
    \big(\fZ_{B_1}\cap (\{0\}\times\fG) \big)_{+}
    := \bigcap_{(t,u)\in\sP}
    \big\{C\in \Matr[Sym] : \Tr C=0, \langle Cu,u\rangle\leq t \big\}.
  \end{align*}
  
  If a matrix $C\in \fG_+$ does not vanish, then at least one of its
  eigenvalues is strictly positive (because all eigenvalues are
  real by symmetry and their sum is $0$). If we denote by $v$ the
  corresponding unit eigenvector, then $\langle C v, v\rangle
  >0$. Since the set of $u_i$'s for which $(t_i,u_i)\in\sP$ is a.s.\
  dense on the unit sphere in $\R^d$, it follows that
  $\langle C u_i, u_i\rangle >0$ for some $i$. Thus,
  $sC\notin \big(\fZ_{B_1}\cap (\{0\}\times\fG) \big)_{+}$ if $s>0$ is
  sufficiently large. Therefore, the convex set
  $\big(\fZ_{B_1}\cap (\{0\}\times\fG) \big)_{+}$ is a.s.~bounded, hence, is a
  compact subset of $\Matr[Sym]$.
  
  As in the previous example, the unbounded component
  $\{0\}\times \Matr[SSym]$ is present in
  $\fZ_{B_1}\cap (\{0\}\times\fG)$ due to the fact that $B_1$ is
  invariant with respect to the orthogonal group $\ortho$ which is a
  Lie subgroup of $\mathbb{SL}_d$. Since arbitrarily large scalings
  are not allowed in $\mathbb{SL}_d$, the random closed set
  $\fZ_{B_1}\cap (\{0\}\times\fG)$ is a.s.~bounded on the complement
  to $\{0\}\times \Matr[SSym]$.
\end{example}

\begin{example}
  Let $\GG=\ortho$ be the orthogonal group. As has already been
  mentioned, the corresponding Lie algebra $\fG=\Matr[SSym]$ is the
  $d(d-1)/2$-dimensional subspace of $\Matr$, consisting of all skew
  symmetric matrices. If $\TT=\R^d$, then $\fC_{\HH}=\R^d\times \fG$
  and
  \begin{displaymath}
    \fZ_{K}\cap (\R^d\times \fG)=\bigcap_{(t,\eta,u)\in\sP_K}
    \Big\{(x,C)\in\R^d\times \Matr[SSym]:
    \langle C\eta,u\rangle +\langle x,u\rangle \leq t \Big\}.
  \end{displaymath}
  In the special case $d=2$ the Lie algebra $\fG$ is one-dimensional
  and is represented by the matrices
  \begin{displaymath}
    C=
    \begin{pmatrix}
      0 & c\\
      -c & 0
    \end{pmatrix}, \quad c\in\R.
  \end{displaymath}
  Write $\eta:=(\eta',\eta'')$ and $u:=(u',u'')$. Then, with $\cong$
  denoting the isomorphism of
  $\R^2\times \mathsf{M}_2^{\mathrm{SSym}}$ and $\R^2\times \R$, we
  can write
  \begin{align}
    \fZ_{K}\cap \big(\R^2\times \mathsf{M}_2^{\mathrm{SSym}}\big)
    &\cong \bigcap_{(t,\eta,u)\in\sP_K}
      \Big\{(x,c)\in \R^2\times\R: c(\eta''u'-\eta'u'')
      +\langle x,u\rangle \leq t\Big\}\notag\\
    &=\bigcap_{(t,\eta,u)\in\sP_K}
      \Big\{(x,c)\in \R^2\times\R: \langle x,u\rangle \leq
      t-c[u,\eta]\Big\}
      \label{eq:revision_label1},
  \end{align}
  where $[u,\eta]$ is the (signed) area of the parallelogram spanned
  by $u$ and $\eta$. Therefore,
  $\fZ_{K}\cap \big (\R^2\times \mathsf{M}_2^{\mathrm{SSym}}\big)$ is
  (isomorphic to) the zero cell of a hyperplane tessellation
  $H_{(u,[u,\eta])}(t)$, $(t,\eta,u)\in\sP_K$ in $\R^2\times\R$.

  Let $K=[-1,1]^2$ be a square in $\R^2$. Let $e_1,e_2$ be the
  standard basis of $\R^2$, and let $u_1=e_1,u_2=e_2,u_3=-e_1,u_4=-e_2$
  be the unit normal vectors to the sides of $K$.  Then
    $\sP_K=\{(t_i,\eta_i,u_K(\eta_i)),i\geq1\}$, where $\{t_i,i\geq1\}$
    is a homogeneous Poisson process on $\R_+$ of intensity $2$ (which
    is the ratio of the perimeter of $K$ and its area), and
    $(\eta_i,u_K(\eta_i))$ are i.i.d.\ pairs composed of $\eta_i$
    uniformly distributed on $\partial K$ and $u_K(\eta_i)$ being the unit
    outer normal to $K$ at $\eta_i$. Let $\sP_1,\dots,\sP_4$ be 
    independent 
    Poisson processes on $\R_+\times[-1,1]$ obtained by letting $\sP_j$ consist of points
    $(t_i,\bar\eta_i)$, $i\geq1$, such that
    $(t_i,\eta_i,u_K(\eta_i))\in\sP_K$ with $u_K(\eta_i)=u_j$, and
    $\bar\eta_i$ is the random component of $\eta_i$, e.g., 
    $\eta_i=(1,\bar\eta_i)$ if $u_K(\eta_i)=u_1$. Note that the
    intensity of $\sP_j$ is $1/2$. In view of the symmetry of
    $\bar\eta_i$, \eqref{eq:revision_label1} implies
    \begin{displaymath}
      \fZ_{K}\cap (\R^2\times \mathsf{M}_2^{\mathrm{SSym}})
      \cong\bigcap_{j=1}^4 \bigcap_{(t,y)\in\sP_j}
      \Big\{(x,c)\in \R^2\times\R:cy+\langle x,u_j\rangle \leq t\Big\}.
    \end{displaymath}

%   \Blue{ Then, 
%   \begin{multline*}
%   \Nor(K)=\left(\{1\}\times[-1,1]\times \{u_1\}\right)\cup \left(\{-1\}\times[-1,1]\times \{u_3\}\right)\\
%   \cup
%   \left([-1,1]\times \{1\}\times \{u_2\}\right)\cup \left([-1,1]\times \{-1\}\times \{u_4\}\right).
%   \end{multline*}
%   Thus, with $\mathcal{P}_1,\ldots,\mathcal{P}_4$ denoting independent homogeneous 
%   Poisson processes on $\R_+\times [-1,1]$ of intensity $1/4$ (recall that the
%   area of $K$ is $4$), we have the representation
%   \begin{multline*}
%   \mathcal{P}_K=\left(\{1\}\times \mathcal{P}_1\times \{u_1\}\right)\cup \left(\{-1\}\times \mathcal{P}_2\times \{u_3\}\right)\cup \\
%   \left(\mathcal{P}_3\times\{1\}\times \{u_2\}\right)\cup \left(\mathcal{P}_4\times\{-1\}\times \{u_4\}\right)
% \end{multline*}
%   Hence,~\eqref{eq:revision_label1} implies
%   \begin{displaymath}
%     \fZ_{K}\cap (\R^2\times \mathsf{M}_2^{\mathrm{SSym}})\cong\bigcap_{i=1}^4 \bigcap_{(t,y)\in\sP_i}
%     \Big\{(x,c)\in \R^2\times\R:c y+\langle x,u_i\rangle \leq t\Big\}.
%   \end{displaymath}}

  In the special case $\TT=\{0\}$ we can set $x=0$, so that
  \begin{displaymath}
    \fZ_{K}\cap (\{0\} \times \mathsf{M}_2^{\mathrm{SSym}})
    \cong\{0\}\times \bigcap_{j=1}^4 \bigcap_{(t,y)\in\sP_j}
    \big\{c\in\R:c y\leq t \big\}.
  \end{displaymath}
  An easy calculation shows that the
  double intersection above is a segment $[-\zeta',\zeta'']$, where
  $\zeta'$ and $\zeta''$ are two independent exponentially distributed random
  variables of mean one.
\end{example}

\begin{example}[Scaling by constants]
  Let $\HH$ be the product of $\TT=\R^d$ and the family
  $\GG=\{e^{r}\II:r\in\R\}$ of scaling transformations, so that
  $\fC_{\HH}=\R^d\times \{r\II:r\in\R\}$. Then, with $\cong$ denoting
  the natural isomorphism between $\fC_{\HH}$ and $\R^d\times \R$,
  \begin{align*}
    \fZ_{K}\cap \fC_{\HH}
    &\cong\bigcap_{(t,\eta,u)\in\sP_K}
    \Big\{(x,r)\in\R^{d}\times\R:
    r\langle \eta,u\rangle + \langle x,u\rangle \leq t\Big\}\\
    &=\bigcap_{(t,\eta,u)\in\sP_K}
    \Big\{(x,r)\in\R^{d}\times\R: rh(K,u) + \langle x,u\rangle \leq t\Big\}\\
    &=\bigcap_{(t,\eta,u)\in\sP_K}
    \Big\{(x,r)\in\R^{d}\times\R: h(rK+x,t^{-1}u)\leq 1\Big\}\\    
    &=\bigcap_{(t,\eta,u)\in\sP_K}
      \Big\{(x,r)\in\R^{d}\times\R: t^{-1}u\in (rK+x)^{o}\Big\}\\
    &=\Big\{(x,r)\in\R^{d}\times\R:
      \conv (\{t^{-1}u: (t,\eta,u)\in\sP_K\})\subseteq (rK+x)^{o}\Big\}.
  \end{align*}
  The set $\conv (\{t^{-1}u: (t,\eta,u)\in\sP_K\})$ has been studied
  in \cite{MarMol:2021}, in particular, the polar set to this hull is
  the zero cell $Z_K$ of the Poisson hyperplane tessellation in $\R^d$,
  whose intensity measure is the product of the Lebesgue measure
  (scaled by $V_d^{-1}(K)$) and the surface area measure
  $S_{d-1}(K,\cdot)=\Theta_{d-1}(K,\R^d\times\cdot)$ of $K$. Thus, we
  can write
  \begin{equation}\label{eq:scalings_translatins}
    \fZ_{K}\cap \fC_{\HH}\cong \{(x,r)\in\R^{d}\times\R:
    rK+x\subseteq Z_K\}.
  \end{equation}

  If $K$ is the unit Euclidean ball $B_1$, then
  \eqref{eq:scalings_translatins} can be recast as
  \begin{displaymath}
    \fZ_{B_1}\cap \fC_{\HH}\cong
    \big\{(x,r)\in \R^d\times\R: x+rB_1\subset Z \big\},
  \end{displaymath}
  where $Z=\bigcap_{i\geq1} H^-_{u_i}(t_i)$ is the zero cell
  generated by the stationary Poisson hyperplane process
  $\{H_{u_i}(t_i),i\geq1\}$ in $\R^d$. For every $r_0\geq 0$, the
  section of $\fZ_{B_1}\cap \fC_{\HH}$ by the hyperplane $\{r=r_0\}$
  is the set $\{x\in\R^d: x+B_{r_0} \subset Z\}$. If $r_0<0$, then the
  section by $\{r=r_0\}$ is the Minkowski sum $Z+B_{-r_0}$.
  % The set
  % $\fZ_{B_1}\cap \fC_{\HH}$ can also be considered as the zero cell of
  % the Poisson hyperplane tessellation in $\R^{d+1}$ with the
  % directional measure being the product of the surface area measure of
  % the unit ball (scaled by its volume) and the Dirac measure
  % $\delta_{-1}$.
  % PREVIOUS VARIANT
  % If $K$ is the unit Euclidean ball $B_1$, then
  % \eqref{eq:scalings_translatins} can be recast as
  % \begin{displaymath}
  %   \fZ_{B_1}\cap \fC_{\HH}\cong
  %   \Big\{(x,r)\in \R^d\times\R: x\in \bigcap_{i\geq1} H^-_{u_i}(t_i-r)\Big\},
  % \end{displaymath}
  % where $\{H_{u_i}(t_i),i\geq1\}$ is a stationary Poisson hyperplane
  % tessellation in $\R^d$. For every $r_0\geq 0$, the section of
  % $\fZ_{B_1}\cap \fC_{\HH}$ by the hyperplane $\{r=r_0\}$ is the set
  % $\{x\in\R^d: x+B_{r_0} \subset Z\}$, where $Z$ is the zero cell of
  % the aforementioned tessellation. If $r_0<0$, then the section by
  % $\{r=r_0\}$ is the Minkowski sum $Z+B_{-r_0}$. The set
  % $\fZ_{B_1}\cap \fC_{\HH}$ can also be considered as the zero cell of
  % the Poisson hyperplane tessellation in $\R^{d+1}$ with the
  % directional measure being the product of the surface area measure of
  % the unit ball (scaled by its volume) and the Dirac measure
  % $\delta_{-1}$.
\end{example}

\begin{example}[Diagonal matrices]
  Let $\GG$ be the group of diagonal matrices with positive entries
  given by $\diag(e^{z_1},\dots,e^{z_d})$ for
  $z=(z_1,\dots,z_d)\in\R^d$.
  % If $\TT=\R^d$, then
  % $\fC_{\HH} \cong \R^d\times\R^d$. If $K$ is the unit ball, we get
  % \begin{displaymath}
  %   \fZ_{B_1}\cap\fC_{\R^d\times \GG}\cong\Big\{(x,z)\in \R^{d}\times\R^d:
  %   x\in \bigcap_{(t,u)\in\sP} H^-_{u}\Big(t-\sum_{j=1}^d
  %   z_ju_{j}^2\Big)\Big\}.
  % \end{displaymath}
  % If $\TT=\{0\}$ and $K$ is arbitrary, then
  If $\TT=\{0\}$, then $\fC_{\HH} \cong \{0\}\times\R^d$ and 
  \begin{align*}
    \fZ_{K}\cap\fC_{\HH}
    &=\{0\}\times\bigcap_{(t,\eta,u)\in\sP_K} \big\{z\in \R^{d}:
         \langle \diag(z)\eta,u\rangle \leq t\big\}\\
    &=\{0\}\times \bigcap_{(t,\eta,u)\in\sP_K} \big\{z\in \R^{d}:
         \langle z,(\diag(\eta)u)\rangle \leq t\big\},
  \end{align*}
  where $\diag(\eta) u$ is the vector given by componentwise products
  of $\eta$ and $u$. Thus, the above intersection is the zero cell of
  a Poisson tessellation in $\R^d$ whose directional measure is
  obtained as the pushforward of $\Theta_{d-1}(K,\cdot)$ under the
  map $(x,u)\mapsto \diag(x)u$. If $K$ is the unit ball, then this
  directional measure is the pushforward of the uniform distribution
  on the unit sphere under the map which transforms $x$ to the vector
  composed of the squares of its components. By
  Proposition~\ref{prop:rec_cone}
  \begin{displaymath}
    {\rm rec}(\fZ_{B_1}\cap\fC_{\{0\}\times \GG})
    \cong \{0\}\times \bigcap_{u\in B_1}
    \big\{z\in\R^d:\langle z,\diag(u)u\rangle\leq 0 \big\}
    =\{0\}\times (-\infty,0]^d.
  \end{displaymath}
\end{example}

\begin{example}[Random cones and spherical polytopes]\label{ex:cones}
  Assume that $K$ is the closed unit upper half-ball $B_1^{+}$ defined
  at \eqref{eq:upper-ball}, $\TT=\{0\}$ and $\GG=\mathbb{SO}_d$, so
  that $\fC_{\HH}=\{0\}\times \Matr[SSym]$. If $\Xi_n$ is a sample
  from the uniform distribution in $K$, then $\pi(\Xi_n)$ is a sample
  from the uniform distribution on the half-sphere
  $\Sphereplus$, where $\pi(x)=x/\|x\|$ for $x\neq
  0$. Indeed, for a Borel set $A\subset \SS^{d-1}_{+}$,
  \begin{displaymath}
    \P\{\pi(\xi_1)\in A\}=\P\{\xi_1/\|\xi_1\|\in A\}
    =\P\{\xi_1\in {\rm pos}(A)\}
    =\frac{V_d({\rm pos}(A)\cap B_1^{+})}{V_d(B_1^{+})}
    =\frac{\mathcal{H}_{d-1}(A)}{\mathcal{H}_{d-1}(\SS^{d-1}_{+})},
  \end{displaymath}
  where $\mathcal{H}_{d-1}$ is the $(d-1)$-dimensional Hausdorff
  measure. According to \eqref{eq:cones2},
  \begin{displaymath}
    Q_n\cap \Sphere=\conv_{K,\HH}(\Xi_n)\cap\Sphere
    = {\rm pos}(\Xi_n) \cap \Sphere ={\rm pos}(\pi(\Xi_n))\cap \Sphere
  \end{displaymath}
  is a closed random spherical polytope obtained as the spherical hull
  of $n$ independent points, uniformly distributed on
  $\SS^{d-1}_{+}$. This object has been intensively studied in
  \cite{kab:mar:tem:19}. Let $T_n:\R^d\to\R^d$ be a linear mapping
  \begin{displaymath}
    T_n(x_1,x_2,\ldots,x_d)=(n x_1,x_2,\ldots,x_d),\quad n\in\NN.
  \end{displaymath}
  Theorem~2.1 in \cite{kab:mar:tem:19} implies that the sequence of
  random closed cones $(T_n({\rm pos}(\Xi_n)))_{n\in\NN}$ converges in
  distribution in the space of closed subsets of $\R^d$ endowed with
  the Fell topology to a closed random cone whose intersection with
  affine hyperplane $\{(x_1,x_2,\dots,x_d)\in\R^d:x_1=1\}$ is the convex set
  $(\conv(\widetilde{\sP}),1)$, where $\widetilde{\sP}$ is a Poisson
  point process on $\R^{d-1}$ with the intensity measure
  \begin{equation}
    \label{eq:poisson_intensity_sphere_d-1}
    x\mapsto c_d\|x\|^{-d},\quad x\in \R^{d-1}\setminus \{0\},
  \end{equation}
  and an explicit positive constant $c_d$. The following arguments
  show that it is possible to establish an isomorphism between the
  positive dual cone
  $\{x\in\R^d:\langle x,\xi_k\rangle\geq 0,k=1,\dots,n\}$ to the cone
  ${\rm pos}(\Xi_n)$ and the set $\XX_{K,\HH}(\Xi_n)$ defined at
  \eqref{eq:4}, so that our limit theorem yields the limit for this
  normalised dual cone.
  
  Denote by $e_1,\dots,e_d$ standard basis vectors. Since
  $\langle C\eta,u\rangle=0$ for all $(t,\eta,u)\in\sP_K$ with
  $u\in \SS^{d-1}_{+}$ and $C\in\Matr[SSym]$, we need only to consider
  $(t,\eta,u)\in\sP_K$ such that $u=-e_1$, meaning that $\eta$ lies on
  the flat boundary part of $B_1^+$, so that
  \begin{align*}
    \fZ_{B_1^{+}}\cap(\{0\}\times \Matr[SSym])
    &=\{0\}\times \bigcap_{(t,\eta,u)\in\sP_{B_1^{+}}}
    \Big\{C\in\Matr[SSym]: \langle C\eta,u\rangle \leq t\Big\}\\
    &=\{0\}\times\bigcap_{(t,\eta,-e_1)\in\sP_{B_1^{+}}}
    \Big\{C\in\Matr[SSym]: \langle C\eta,-e_1\rangle \leq t\Big\}\\
    &=\{0\}\times\bigcap_{(t,\eta,-e_1)\in\sP_{B_1^{+}}}
    \Big\{C\in\Matr[SSym]: \langle \eta,Ce_1\rangle \leq t\Big\}.
  \end{align*}
  Note that every skew-symmetric matrix can be uniquely decomposed
  into a sum of a skew-sym\-metric matrix with zeros in the first row
  and the first column and a skew symmetric matrix with zeros
  everywhere except the first row and the first column. This corresponds
  to the direct sum decomposition of the space of skew symmetric
  matrices $\Matr[SSym]:=V_1\oplus V_2$, where
  $V_1\cong \mathsf{M}^{\mathrm{SSym}}_{d-1}$. For every
  $(t,\eta,-e_1)\in\sP_{B_1^{+}}$ and $C\in V_1$, we obviously have
  $\langle \eta,Ce_1\rangle=0$. Thus,
  \begin{displaymath}
    \fZ_{B_1^{+}}\cap \big(\{0\}\times \Matr[SSym]\big)
    \cong \{0\}\times\Big(\mathsf{M}^{\mathrm{SSym}}_{d-1}
      \oplus \bigcap_{(t,\eta,-e_1)\in\sP_{B_1^{+}}}
      \big\{C\in V_2: \langle \eta,Ce_1\rangle \leq t\big\}\Big).
  \end{displaymath}
  The fact that $\fZ_{B_1^{+}}\cap(\{0\}\times \Matr[SSym])$ contains
  the subspace $V_1$ has the following interpretation. It is known
  that the exponential map from $\Matr[SSym]$ to $\mathbb{SO}_d$ is
  surjective, that is, every orthogonal matrix with determinant one
  can be represented as the exponent of a skew-symmetric matrix, see
  Corollary~11.10 in \cite{Hall:2003}. The image $\exp(V_1)$ is
  precisely the set of orthogonal matrices with determinant one and
  for which $e_1$ is a fixed point. This set is a subgroup of
  $\mathbb{SO}_d$ which is isomorphic to $\mathbb{SO}_{d-1}$, and
  $B_1^{+}$ is invariant with respect to all transformations from
  $\exp(V_1)$. The set $\exp(V_2)$ is not a subgroup of
  $\mathbb{SO}_d$ but is a smooth manifold of dimension $d-1$. Note
  that the above construction is the particular case of the well-known
  general concept of quotient manifolds in Lie groups, see Chapter
  11.4 in \cite{Hall:2003}.
  
  There is a natural isomorphism $\phi:\{0\}\times V_2\to \R^{d-1}$
  which sends $(0,C)\in \{0\}\times V_2$ to the vector
  $\phi(C)\in\R^{d-1}$ which is the first column of $C$ with the first
  component (it is always zero) deleted. Moreover, if
  $(t,\eta,-e_1)\in\sP_{B_1^{+}}$, then $\eta$ is necessarily of the
  form $\eta=(0,\eta')$, where $\eta'\in B'_1$ and $B'_1$ is a
  $(d-1)$-dimensional centred unit ball. It can be checked that the Poisson
  process
  $\{(t^{-1}\eta')\in \R^{d-1}\setminus\{0\} :
  (t,\eta,-e_1)\in\sP_{B_1^{+}}\}$ has
  intensity~\eqref{eq:poisson_intensity_sphere_d-1}.  Summarising,
  \begin{displaymath}
    \phi \big(\fZ_{B_1^{+}}\cap(\{0\}\times V_2) \big)
    =\bigcap_{(t,\eta,-e_1)\in\sP_{B_1^{+}}}\big\{x\in \R^{d-1}:
    \langle \eta',x\rangle \leq t\big\}=:\widetilde{Z}_0,
  \end{displaymath}
  is the zero cell of the Poisson hyperplane tessellation
  $\{H^{-}_{\eta'}(t):(t,\eta,-e_1)\in\sP_{B_1^{+}}\}$ of
  $\R^{d-1}$. Remarkably, the polar set to $\widetilde{Z}_0$ is the
  convex hull of $\widetilde{\sP}$. 
\end{example}

\section{Appendix}
\label{sec:appendix}

The subsequent presentation concerns random sets in Euclidean space
$\R^d$ of generic dimension $d$. These results are applied in the main
part of this paper to random sets of affine transformations, which are
subsets of the space $\R^d\times\Matr$. This latter space can be
considered an Euclidean space of dimension $d+d^2$.

Let $\sF^d$ be the family of closed sets in $\R^d$. Denote by $\sC^d$
the family of nonempty compact sets and by $\sK^d$ the family of
nonempty compact convex sets. The family of compact convex sets
containing the origin is denoted by $\sK^d_0$, while $\sK^d_{(0)}$ is
the family of compact convex sets which contain the origin in their
interiors. Each set from $\sK^d_{(0)}$ is a convex body (a compact
convex set with nonempty interior).

The family $\sF^d$ is endowed with the Fell topology, whose base
consists of finite intersections of the sets
$\{F:F\cap G\neq\emptyset\}$ and $\{F:F\cap L=\emptyset\}$ for all
open $G$ and compact $L$. The definition of the Fell topology and its
basic properties can be found in Section~12.2 of \cite{sch:weil08} or
Appendix~C in \cite{mol15}. Note that $F_n\to F$ in the
Fell topology (this will be denoted by $F_n\toFn F$) if and only if
$F_n$ converges to $F$ in the Painlev\'e--Kuratowski sense, that is,
$\limsup F_n=\liminf F_n=F$. Recall that $\limsup F_n$ is the set of
all limits of convergent subsequences $x_{n_k}\in F_{n_k}$, $k\geq1$,
and $\liminf F_n$ is the set of limits of convergent sequences
$x_n\in F_n$, $n\geq1$.
% The space $\sF^d$ is compact in the Fell
% topology, see Theorem 12.2.1 in \cite{sch:weil08}.
The family $\sC^d$ is endowed with the topology generated by the
Hausdorff metric which we denote by $d_{\rm H}$.
% The topology induced
% by $d_{\rm H}$ on $\sK^d$ is exactly the Painlev\'e--Kuratowski
% topology, that is, the topology induced on $\sK^d\subset\sF^d$ by the
% Fell topology on $\sF^d$, see Theorem~1.8.8 in~\cite{schn2}. In
% comparison, the topology induced by $d_{\rm H}$ on $\sC^d$ is strictly
% finer, then the topology induced on $\sC^d$ by the Fell topology, see
% Theorem~12.3.2 in \cite{sch:weil08}.

It is easy to see that the convergence $(F_n\cap L)\toFn (F\cap L)$
as $n\to\infty$ for each compact set $L$ implies the Fell convergence
$F_n\toFn F$. The inverse implication is false in general, since the
intersection operation is not continuous, see Theorem~12.2.6
in~\cite{sch:weil08}.  The following result establishes a kind of
continuity property for the intersection map. A closed set $F$ is said
to be {\it regular closed} if it coincides with the closure of its
interior. The empty set is also considered regular closed.  A nonempty
closed convex set is regular closed if and only if its interior is not
empty.

\begin{lemma}
  \label{lemma:int-det}
  Let $(F_n)_{n\in\NN}$ and $F$ be closed sets such that $F_n\toFn F$
  as $n\to\infty$, and let $L$ be a closed set in $\R^d$. Assume that one
  of the following conditions hold:
  \begin{enumerate}[(i)]
  \item $F\cap L$ is regular closed; 
  \item the  sets $F$ and $L$ are convex,
    $0\in\Int F$ and $0\in L$.
  \end{enumerate}
  Then $(F_n\cap L)\toFn (F\cap L)$ as $n\to\infty$.
\end{lemma}
\begin{proof}
  By Theorem~12.2.6(a) in~\cite{sch:weil08}, we have
  \begin{displaymath}
    \limsup(F_n\cap L)\subset (F\cap L).
  \end{displaymath}
  If $F$ is empty, this finishes the proof. Otherwise, it suffices to
  show that $(F\cap L)\subset \liminf (F_n\cap L)$, assuming that $F$
  is not empty, so that $F_n$ is also nonempty for all sufficiently
  large $n$.

  (i) For every $x\in\Int(F\cap L)$, there exists a sequence
  $x_n\in F_n$, $n\geq1$, such that $x_n\to x$ and $x_n\in L$ for all
  sufficiently large $n$. Thus,
  $\Int(F\cap L)\subset \liminf(F_n\cap L)$ and therefore
  $$
  F\cap L= \cl (\Int (F\cap L))\subset \liminf(F_n\cap L),
  $$
  where for the equality we have used that $F\cap L$ is regular
  closed, and for the inclusion that the lower limit is always a
  closed set.

  (ii) Note that
  \begin{equation}\label{lem:intersection_proof1}
  \cl ((\Int F)\cap L)=F\cap L.
  \end{equation}
  Indeed, if $x\in (F\cap L)\setminus \{0\}$, then convexity of
  $F\cap L$ and $0\in F\cap L$ imply that
  $x_n:=(1-\frac{1}{n})x\in (\Int F)\cap L$, for all $n\in\NN$. Since
  $x_n\to x$, we obtain $x\in \cl ((\Int F)\cap L)$. Obviously,
  $\{0\}\in\cl ((\Int F)\cap L)$. Thus,
  $F\cap L \subseteq \cl ((\Int F)\cap L)$. The inverse inclusion
  holds trivially. Taking into account~\eqref{lem:intersection_proof1}
  and that $\liminf$ is a closed set it suffices to show that
  \begin{equation}\label{lem:intersection_proof2}
  (\Int F)\cap L\subset \liminf(F_n\cap L).
  \end{equation}
  Assume that $x\in (\Int F)\cap L$. Pick a small enough $\eps>0$ and
  a sufficiently large $R>0$ such that $x+B_\eps\subset F\cap B_R$.
  Since $F\cap B_R$ is convex and contains the origin in the interior,
  it is regular closed. Thus, by part (i),
  $F_n\cap B_R\toFn F\cap B_R$. By Theorem 12.3.2 in
  \cite{sch:weil08}, we also have $F_n\cap B_R\toHn F\cap B_R$. In
  particular, there exists $n_0\in\NN$ such that
  $F\cap B_R\subset (F_n\cap B_R)+B_{\eps/2}$, for $n\geq n_0$, and,
  thereupon, $x+B_{\eps/2}\subset F_n$. Hence $x\in F_n\cap L$ for
  all $n\geq n_0$. Thus, \eqref{lem:intersection_proof2} holds.
\end{proof}

The following result establishes continuity properties of the polar
transform $L\mapsto L^o$ defined by \eqref{eq:polar} on various
subfamilies of closed convex sets which contain the origin.  It
follows from Theorem~4.2 in \cite{mol15} that the polar map
$L\mapsto L^o$ is continuous on $\sK^d_{(0)}$ in the Hausdorff metric,
equivalently, in the Fell topology. While $L^o$ is compact if $L$
contains the origin in its interior, $L^o$ is not necessarily bounded
for $L\in\sK^d_0\setminus \sK^d_{(0)}$.  Recall that $\dom(L)$ denotes
the set of $u\in\R^d$ such that $h(L,u)<\infty$.

\begin{lemma}
  \label{lemma:polar-map}
  The following facts hold. 
  \begin{enumerate}[(i)]
  \item Let $L$ and $L_n$, $n\in\NN$, be closed convex sets which contain
    the origin. Assume that $\dom(L_n)=\dom(L)$ is closed for all $n\in\NN$,
    and $h(L_n,u)\to h(L,u)$ as $n\to\infty$, uniformly over
    $u\in \dom(L)\cap\Sphere$. Then $L_n^o\to L^o$ in the Fell topology. 
  \item The polar transform is continuous as a map from
    $\sK^d_0$ with the Hausdorff metric to $\sF^d$ with the Fell
    topology. 
  \item The polar transform is continuous as the map from the family
    of closed convex sets which contain the origin in their interior
    with the Fell topology to $\sK^d_0$ with the Hausdorff metric.
  \item The polar transform is continuous as the map from
    $\sK^d_{(0)}$ to $\sK^d_{(0)}$, where both spaces are equipped
    with the Hausdorff metric.  
  \end{enumerate}
\end{lemma}
\begin{proof}  
  (i) Consider a sequence $(x_{n_k})_{k\in\NN}$ such that
  $x_{n_k}\in L^o_{n_k}$, $k\in\NN$, and $x_{n_k}\to x$ as
  $k\to\infty$. Assume that $x\notin L^o$.  If $h(L,x)=\infty$, that
  is, $x\in (\dom(L))^c$, then also $x_{n_k} \in (\dom(L))^c$ for all
  sufficiently large $k$, since the complement to $\dom(L)$ is
  open. Hence, $x_{n_k} \in (\dom(L_{n_k}))^c$ and
  $h(L_{n_k},x_{n_k})=\infty$, meaning that $x_{n_k}\notin
  L_{n_k}^o$. Assume now that $h(L,x)<\infty$ and
  $h(L_{n_k},x_{n_k})<\infty$ for all $k$.  If
  $u,v\in\dom(L)\cap\Sphere$, then $h(L,u)=\langle x,u\rangle$ for
  some $x\in L$, so that
  \begin{displaymath}
    h(L,u)=\langle x,u-v\rangle +\langle x,v\rangle
    \leq \|x\|\|u-v\|+h(L,v).
  \end{displaymath}
  Hence, the support function of $L$ is Lipschitz on
  $\dom(L)\cap\Sphere$ with
  the Lipschitz constant at most
  $c_L:=\sup_{u\in\dom(L)\cap\Sphere} h(L,u)<\infty$.  Since we assume
  $x\notin L^o$, we have $h(L,x)\geq 1+\eps$ for some $\eps>0$. The
  uniform convergence assumption yields that
  \begin{displaymath}
    h(L_{n_k},x_{n_k})\geq h(L,x_{n_k})-\eps/4
    \geq h(L,x)-\eps/4-c_L\|x_{n_k}-x\|\geq 1+\eps/2
  \end{displaymath}
  for all sufficiently large $k$, meaning that
  $x_{n_k}\notin L^o_{n_k}$, which is a contradiction. Hence,
  $\limsup L^o_n\subset L^o$.

  Let $x\in L^o$. Then $h(L,x)\leq1$, so that
  $h(L_n,x)\leq 1+\eps_n$, where $\eps_n\downarrow 0$ as
  $n\to\infty$. Letting $x_n:=x/(1+\eps_n)$, we have that
  $x_n\in L^o_n$ and $x_n\to x$. Thus, $L^o\subset \liminf L^o_n$.

  (ii) If all sets $(L_n)$ and $L$ are compact, then $\dom(L)=\R^d$,
  the convergence in the Hausdorff metric is equivalent to the uniform
  convergence of support functions on $\Sphere$, see Lemma 1.8.14
  in~\cite{schn2}.  Thus, $L_n^o\toFn L^o$ by part (i).

  (iii) Assume that $L_n\toFn L$. In view of
  Lemma~\ref{lemma:int-det}(i), $L_n\cap B_R\toFn L\cap B_R$, for
  every fixed $R>0$, and, therefore, $L_n\cap B_R\toHn L\cap B_R$ by
  Theorem 1.8.8 in~\cite{schn2}. Fix a sufficiently small $\eps>0$
  such that $B_\eps\subset L$. Then $B_{\eps/2}\subset L_n$, for all
  sufficiently large $n$. By part (ii),
  $(L_n\cap B_R)^o\toFn (L\cap B_R)^o$. Since
  $(L_n\cap B_R)^o\subseteq B_{(\eps/2)^{-1}}$, for all sufficiently
  large $n$, $(L_n\cap B_R)^o\toHn (L\cap B_R)^o$, again by
  Theorem~1.8.8 in~\cite{schn2}. Finally, note that
  \begin{align*}
    d_H(L^{o}_n,L^{o})
    &\leq d_H(L^{o}_n,(L_n\cap B_R)^{o})
      +d_H((L_n\cap B_R)^{o},(L\cap B_R)^{o})+
      d_H((L\cap B_R)^{o},L^{o})\\
    &= d_H(L^{o}_n,\conv (L_n^{o}\cup B_{R^{-1}}))
      +d_H((L_n\cap B_R)^{o},(L\cap B_R)^{o})+
      d_H(\conv (L^{o}\cup B_{R^{-1}}),L^{o})\\
    &\leq R^{-1}+d_H((L_n\cap B_R)^{o},(L\cap B_R)^{o})+R^{-1},
  \end{align*}
  where we have used that
  $(A_1\cap A_2)^{o}=\conv (A_1^{o}\cup A_2^{o})$ and
  $B_R^{o}=B_{R^{-1}}$.
  
  (iv) Follows from (iii), since the Fell topology on $\sK^d_{(0)}$
  coincides with the topology induced by the Hausdorff metric.
\end{proof}

A random closed set $X$ is a measurable map from a probability space
to $\sF^d$ endowed with the Borel $\sigma$-algebra generated by the
Fell topology. This is equivalent to the assumption that
$\{X\cap L\neq\emptyset\}$ is a measurable event for all compact sets
$L$. The distribution of $X$ is uniquely determined by its capacity
functional
\begin{displaymath}
  T_X(L)=\Prob{X\cap L\neq\emptyset},\quad L\in\sC^d.
\end{displaymath}
A sequence $(X_n)_{n\in\NN}$ of random closed sets in $\R^d$ converges
in distribution to a random closed set $X$ (notation $X_n\dodn X$) if
the corresponding probability measures on $\sF^d$ (with the Fell
topology) weakly converge. By Theorem~1.7.7 in~\cite{mo1}, this is
equivalent to the pointwise convergence of capacity functionals
\begin{equation}
  \label{eq:3}
  T_{X_n}(L)\to T_X(L) \quad \text{as}\; n\to\infty
\end{equation}
for all $L\in\sC^d$ which satisfy
\begin{equation}
  \label{eq:15}
  \Prob{X\cap L\neq\emptyset}=\Prob{X\cap \Int L\neq\emptyset},
\end{equation}
that is, $T_X(L)=T_X(\Int L)$.  The latter condition means that the
family $\{F\in\sF^d:F\cap L\neq\emptyset\}$ is a continuity set for
the distribution of $X$, and we also say that $L$ itself is a
continuity set.  It suffices to impose \eqref{eq:3} for sets $L$ which
are regular closed or which are finite unions of balls of positive
radii; these families constitute so called convergence determining
classes, see Corollary~1.7.14 in~\cite{mo1}.

\begin{lemma}
  \label{lemma:restriction}
  A sequence of random closed sets $(X_n)_{n\in\NN}$ in $\R^d$
  converges in distribution to a random closed set $X$ if there exists
  a sequence $(L_m)_{m\in\NN}$ of compact sets such that
  $\Int L_m\uparrow\R^d$ and $(X_n\cap L_m)\dodn (X\cap L_m)$ as
  $n\to\infty$ for each $m\in\NN$.
\end{lemma}
\begin{proof}
  We will check \eqref{eq:3}.  Fix an $L\in\sC^d$ such that
  $T_X(L)=T_X(\Int L)$. Pick $m\in\NN$ so large that $L_m$ contains
  $L$ in its interior. We have that
  \begin{displaymath}
    T_{X\cap L_m}(L)=\Prob{X\cap L_m\cap L\neq\emptyset}
    =\Prob{X\cap L_m\cap \Int L\neq\emptyset}=T_{X\cap L_m}(\Int L).
  \end{displaymath}
  Since $(X_n\cap L_m)\dodn (X\cap L_m)$ as $n\to\infty$, we have that
  \begin{multline*}
    T_{X_n}(L)=\Prob{X_n \cap L\neq\emptyset}
    =\Prob{X_n\cap L_m \cap L\neq\emptyset}\\
    \to\Prob{X\cap L_m\cap L\neq\emptyset}=\Prob{X\cap L\neq\emptyset}=T_X(L),
  \end{multline*}
  meaning that $X_n\dodn X$.
\end{proof}

For a random closed set $X$, the functional
\begin{displaymath}
  I_X(L)=\Prob{L\subset X}, \quad L\in\sB(\R^d),
\end{displaymath}
is called the inclusion functional of $X$. While the capacity
functional determines uniquely the distribution of $X$, this is not
the case for the inclusion functional, e.g., the inclusion functional
vanishes if $X$ is a singleton with a nonatomic distribution.

Let $\sE$ be the family of all regular closed convex subsets of $\R^d$
(excluding the empty set), and let $\sE'$ denote
the family of closed complements to all sets from $\sE$. Recall that a
nonempty closed convex set belongs to $\sE$ if and only if its
interior is not empty.

\begin{lemma}
  \label{lemma:bicont}
  The map $F\mapsto \cl(F^c)$ is a bicontinuous (in the Fell topology)
  bijection between $\sE$ and $\sE'$.
\end{lemma}
\begin{proof}
  The map $F\mapsto\cl(F^c)$ is self-inverse on $\sE$, hence a
  bijection. Let us prove continuity.
  % Recall that a map
  % $\phi:\mathcal{F}^d\to\mathcal{F}^d$ is called lower semicontinuous
  % if for any open set $G$ such that $\phi(F)\cap G\neq \varnothing$,
  % the convergence $F_n\toFn F$ implies
  % $\phi(F_n)\cap G\neq \varnothing$ for all sufficiently large $n$,
  % see~\cite[Theorems~12.2.2 and 12.2.5]{sch:weil08}.  Analogously,
  % $\phi$ is said to be upper semicontinuous if for any compact set $L$
  % such that $\phi(F)\cap L=\varnothing$, the convergence $F_n\toFn F$
  % implies $\phi(F_n)\cap L=\varnothing$ for all sufficiently large
  % $n$.  
  It is known that the map $F\mapsto\cl(F^c)$ is lower
  semicontinuous on $\mathcal{F}^d$,
  see~\cite[Theorem~12.2.6(b)]{sch:weil08}.  Thus, it suffices to
  prove its upper semicontinuity on $\sE$ and $\sE'$.

  Assume that $F_n\toFn F$.
%   Let $\cl(F^c)\cap G\neq\emptyset$ for an
%   open set $G$. Then $F\neq \R^d$ and $F^c\cap G\neq\emptyset$,
%   meaning that $G$ is not a subset of $F$. Take $x\in G\setminus
%   F$. There exists an $\eps>0$ such that $x+B_\eps\subset G$ and
%   $(x+B_\eps)\cap F=\emptyset$. Since $F_n$ converges to $F$, we have
%   that $(x+B_\eps)\cap F_n=\emptyset$, for sufficiently large
%   $n$. Thus, $F_n^c\cap G\neq\emptyset$, which means that
%   $\cl(F_n^c)\cap G\neq\emptyset$ for sufficiently large $n$.
% %  This argument also applies if $F_n$ converges to the empty set.
  Suppose that $\cl(F^c)\cap L=\emptyset$ for a nonempty compact set
  $L$ (in this case $F$ is necessarily nonempty). By compactness,
  $\cl(F^c)\cap (L+B_\eps)=\varnothing$, for a sufficiently small
  $\eps>0$. Therefore,
  \begin{equation}
    \label{eq:bicont_lemma_proof1}
    L+B_\eps\subseteq (\cl(F^c))^c=\Int F.
  \end{equation}
  By convexity of $F$, it is possible to replace $L$ with its convex
  hull, so assume that $L$ is convex. Pick a large $R>0$ such that
  $L+B_\eps\subseteq \Int B_R$. From Lemma~\ref{lemma:int-det}(i) and
  using the same reasoning as in the proof of part~(iii) of
  Lemma~\ref{lemma:polar-map} we conclude that
  $$
  F_n\cap B_R \toHn F\cap B_R\quad \text{as}\; n\to\infty.
  $$ 
  Thus, $(F\cap B_R)\subset (F_n\cap B_R)+B_{\eps/2}$ for all
  sufficiently large $n$. In conjunction
  with~\eqref{eq:bicont_lemma_proof1}, this yields
  $L+B_\eps\subset (F_n\cap B_R)+B_{\eps/2}$ for sufficiently large
  $n$. Since $L$ and $F_n\cap B_R$ are convex, we conclude that
  $L\subset \Int (F_n\cap B_R)\subset \Int F_n$. Hence,
  $L\cap \cl(F_n^c)=\emptyset$ for all sufficiently large $n$. This
  observation completes the proof of continuity of the direct mapping.

  It remains to prove upper semicontinuity of the inverse
  mapping. Assume that $\cl(F_n^c)\toFn \cl(F^c)$ as $n\to\infty$,
  with $F_n,F\in\sE$.
  % If $F\cap G\neq\emptyset$ for an open set $G$,
  % then $(\Int F)\cap G\neq \emptyset$ and also
  % $\cl(F^c)\neq\R^d$. Take a point $x\in (\Int F)\cap G$. Then
  % $x\notin \cl(F^c)$, so that $x\notin \cl(F_n^c)$ for all
  % sufficiently large $n$, meaning that $x\in F_n$ and
  % $G\cap F_n\neq\emptyset$, for all sufficiently large $n$.
  Assume that $F\cap L=\emptyset$ for a nonempty compact set $L$.  We
  aim to show that $F_n\cap L=\emptyset$, for all sufficiently large
  $n$.  By compactness
  of $L$,
  % there exists a finite cover of $L$ by balls
  % which do not intersect $F$. Hence
  it suffices to prove the statement
  for $L$ being a closed ball.  Fix a point
  $z\in\Int F\neq \varnothing$ such that $B_{2\eps}(z)\subset F$ for
  some $\eps>0$. Then $B_{\eps}(z)\cap \cl(F^c)=\varnothing$ and the
  convergence $\cl(F_n^c)\toFn \cl(F^c)$ implies
  $B_{\eps}(z)\cap \cl(F_n^c)=\varnothing$ and so
  $B_{\eps}(z)\subset F_n$ for all sufficiently large $n$.  Since $F$
  is convex, there exists a closed half-space $H^{-}$ such that
  $L\subset \Int H^{-}$ and $F\cap H^{-}=\emptyset$, so that
  $H^{-}\subset\cl(F^c)$. Recall that we aim to show that
  $F_n\cap L=\emptyset$, for all sufficiently large $n$. If this is
  not true, then there exists a sequence $(y_n)$ such that
  $y_n\in F_n\cap L$ for infinitely many $n\in\mathbb{N}$. Passing to
  a subsequence, assume that this holds for all $n$ and
  $y_n\to y\in L$ for $y_n\in F_n$. Let $M_n$ be the convex hull of
  $B_\eps(z)$ and $y_n$, which is a regular closed set such that $M_n$
  converges to the convex hull $M$ of $B_\eps(z)$ and $y$. By the
  first part of the lemma, $\cl(M_n^c)\toFn\cl(M^c)$. Since
  $M_n\subset F_n$, we have $\cl(M_n^c)\supset \cl(F_n^c)$. Thus,
  $H^{-}\subset\cl(F^c)\subset \cl(M^c)$.  This is a contradiction
  because $M\cap \partial H^{-}\neq \varnothing$.  Therefore,
  $L\cap F_n=\emptyset$ and the proof is complete.
\end{proof}

The following result derives the convergence in distribution of
random closed convex sets with values in $\sE$ from the convergence of
their inclusion functionals. It provides an alternative proof and an
extension of Proposition~1.8.16 in~\cite{mo1}, which establishes this
fact for random sets with values in $\sK_{(0)}^d$.

\begin{theorem}
  \label{thm:inclusion}
  Let $X$ and $X_n$, $n\in\NN$ be random closed sets in $\R^d$ which
  almost surely take values from the family $\sE$ of nonempty regular
  closed convex sets. If
  \begin{equation}\label{eq:22}
    \Prob{L\subset X_n}\to \Prob{L\subset X}\quad \text{as}\quad n\to\infty
  \end{equation}
  for all regular closed $L\in\sK^d$ such that 
  \begin{equation}
    \label{eq:2}
    \Prob{L\subset X}=\Prob{L\subset \Int X},
  \end{equation}
  then $X_n\dodn X$ as $n\to\infty$.
\end{theorem}
\begin{proof}
  In view of Lemma~\ref{lemma:bicont} it suffices to prove that
  $\cl(X_n^c)\dodn \cl(X^c)$ as $n\to\infty$. Furthermore, since
  regular closed compact sets constitute a convergence determining
  class, see Corollary~1.7.14 in \cite{mo1}, it suffices to check that
  \begin{equation}\label{eq:thm_inclusion_proof1}
    \Prob{\cl(X_n^c)\cap L\neq\emptyset}\to \Prob{\cl(X^c)\cap
      L\neq\emptyset}
    \quad\text{as}\quad n\to\infty,
  \end{equation}
  for all regular closed $L\in\sC^d$, which are continuity sets for
  $\cl(X^c)$. The latter means that
  \begin{equation}\label{eq:thm_inclusion_proof2}
    \Prob{\cl(X^c)\cap L=\emptyset}=\Prob{\cl(X^c)\cap \Int L=\emptyset}.
  \end{equation}
  Fix a regular closed set $L\in\sC^d$ such
  that~\eqref{eq:thm_inclusion_proof2} holds. Since
  \begin{displaymath}
    \Prob{\cl(X^c)\cap L=\emptyset}=\Prob{L\subset \Int X}\quad\text{and}\quad
    \Prob{\cl(X^c)\cap \Int L=\emptyset}=\Prob{\Int L\subset \Int X},
  \end{displaymath}  
  we conclude that 
  $$
  \Prob{L\subset X}\leq \Prob{\Int L\subset \Int X}
  =\Prob{L\subset \Int X}\leq\Prob{L\subset X}.
  $$
  Thus, $L$ satisfies \eqref{eq:2}. 
  
  Let $(\eps_k)_{k\in\NN}$ be a sequence of positive numbers such that
  $\eps_k\downarrow 0$ as $k\to\infty$, and
  \begin{displaymath}
    \Prob{L+B_{\eps_k}\subset X}=\Prob{L+B_{\eps_k}\subset \Int X},\quad k\in\NN.
  \end{displaymath}
  Sending $n\to\infty$ in the chain of inequalities
  $$
  \Prob{L+B_{\eps_k} \subset X_n}\leq \Prob{L\subseteq \Int X_n}
  =\Prob{\cl(X_n^c)\cap L=\emptyset}
  \leq \Prob{L\subset X_n},
  $$
  and using~\eqref{eq:22}, we conclude that
  \begin{multline}\label{eq:thm_inclusion_proof3}
    \Prob{L+B_{\eps_k} \subset X}\leq \liminf_{n\to\infty}\Prob{\cl(X_n^c)\cap
      L=\emptyset}\\
    \leq \limsup_{n\to\infty}\Prob{\cl(X_n^c)\cap L=\emptyset}
    \leq \Prob{L\subset X}.
  \end{multline}
  Since
  $$
  \Prob{L+B_{\eps_k} \subset X}\uparrow \Prob{L\subset \Int X}
  =\Prob{L\subset X}\quad\text{as}\quad k\to\infty,
  $$
  the desired convergence~\eqref{eq:thm_inclusion_proof1} follows upon
  sending $k\to\infty$ in~\eqref{eq:thm_inclusion_proof3}.
\end{proof}  

If $F$ is an arbitrary closed set, then, in general, the convergence
$X_n\dodn X$ does not imply the convergence of $X_n\cap F$ to
$X\cap F$. The latter is equivalent to the convergence of the capacity
functionals of $X_n\cap F$ on sets $L\in\sC^d$ such that
\begin{displaymath}
%  \label{eq:8}
  \Prob{(X\cap F)\cap L\neq\emptyset}=\Prob{(X\cap F)\cap \Int
    L\neq\emptyset}. 
\end{displaymath}
At a first glance, the aforementioned implication looks plausible
since the capacity functional of $X_n\cap F$ on $L$ is just the
capacity function of $X_n$ on $F\cap L$. However, from the convergence
$X_n\dodn X$ we can only deduce the convergence of their capacity
functionals on sets $F\cap L$ under condition that
\begin{displaymath}
  \Prob{X\cap F\cap L\neq\emptyset}=
  \Prob{X\cap \Int(F\cap L)\neq\emptyset}.
\end{displaymath}
This latter is too restrictive if $F$ has empty interior. The
following result relies on an alternative argument in order to
establish convergence in distribution of random sets intersected with
a deterministic closed convex set containing the origin.

\begin{lemma}
  \label{lemma:intersection}
  Let $X$ and $X_n$, $n\in\NN$, be random closed convex sets. Assume that
  $\Prob{0\in X}=\Prob{0\in \Int X}>0$ and
  $\Prob{0\in X_n}=\Prob{0\in \Int X_n}>0$ for all sufficiently large
  $n$. Assume that \eqref{eq:22} holds for all $L\in\sK^d$ satisfying
  \eqref{eq:2}. Let $F$ be a closed convex set which contains the origin. Then
  \begin{equation}
    \label{eq:10}
    \Prob{X_n\cap F\cap L\neq\emptyset, 0\in X_n}
    \to \Prob{X\cap F\cap L\neq\emptyset, 0\in X}\quad \text{as}\quad n\to\infty
  \end{equation}
  for each compact set $L$ in $\R^d$ such that
  \begin{equation}\label{eq:108}
    \Prob{(X\cap F)\cap L\neq\emptyset, 0\in X}
    =\Prob{(X\cap F)\cap \Int L\neq\emptyset, 0\in X}.
  \end{equation}
\end{lemma}
\begin{proof}
  Let $Y_n$ (respectively, $Y$) be a random closed convex set
  which has the conditional distribution of $X_n$ given that
  $0\in\Int X_n$ (respectively, $0\in\Int X$). Note that the
  conditional distribution does not change if we replace the
  conditions by $0\in X_n$ and $0\in X$.
  % Define
  % the following auxiliary random closed sets
  % \begin{displaymath}
  %   Y_n:=
  %   \begin{cases}
  %     X_n,&\text{if }0\in\Int X_n,\\
  %     \varnothing,&\text{if }0\notin \Int X_n;
  %   \end{cases}
  %   \quad\text{and}\quad
  %   Y:=
  %   \begin{cases}
  %     X,&\text{if }0\in\Int X,\\
  %     \varnothing,&\text{if }0\notin \Int X.
  %   \end{cases}
  % \end{displaymath}
  By construction, random closed convex sets $Y_n$ and $Y$ almost
  surely belong to the family $\sE$. Let us show with the help of
  Theorem~\ref{thm:inclusion}, that $Y_n\dodn Y$ as $n\to\infty$. Let
  $L$ be a nonempty compact set such that
  \begin{equation}
    \label{eq:20}
    \Prob{L\subset Y}=\Prob{L\subset \Int Y}.
  \end{equation}  
  The latter is equivalent to
  \begin{displaymath}
    \Prob{L\subset X,0\in\Int X}=\Prob{L\subset\Int X,0\in\Int X},
  \end{displaymath}  
  and, since $\Prob{0\in X}=\Prob{0\in\Int X}$, \eqref{eq:20} is also
  equivalent to 
  \begin{displaymath}
   \Prob{L\subset X,0\in X}=\Prob{L\subset\Int X,0\in\Int X}.
  \end{displaymath} 
  Finally, by convexity of $X$ we see that~\eqref{eq:20} is the same as
  \begin{displaymath}
    \Prob{\conv(L\cup\{0\})\subset X}=\Prob{\conv(L\cup\{0\})\subset
      \Int X}.
  \end{displaymath}
  Thus, if a nonempty compact set $L$ satisfies~\eqref{eq:20},
  then $\conv(L\cup\{0\})\in\sK^d$ satisfies~\eqref{eq:2} and we 
  conclude that
  \begin{multline*}
    \Prob{L\subset Y_n}
    =\frac{\Prob{L\subset X_n,0\in X_n}}{\Prob{0\in X_n}}
    =\frac{\Prob{\conv(L\cup\{0\})\subset X_n}}{\Prob{0\in X_n}}\\
      \to \frac{\Prob{\conv(L\cup\{0\})\subset X}}{\Prob{0\in X}}
        =\frac{\Prob{L\subset X,0\in X}}{\Prob{0\in X}}
        =\Prob{L\subset Y}
    \quad\text{as}\quad n\to \infty,
  \end{multline*}
  where the convergence of the numerators (respectively, denominators)
  follows from \eqref{eq:22} with $L$ replaced by $\conv(L\cup\{0\})$
  (respectively, by $\{0\}$).
  
  Theorem~\ref{thm:inclusion} yields that $Y_n\dodn Y$ as
  $n\to\infty$. 
  %Note that $Y$ is a.s.~either empty or contains $0$ in
  %the interior. Thus, $Y\cap F$ is a.s.~either empty (thus, regular
  %closed) or $Y$ contains $0$ in the interior. 
  By Lemma~\ref{lemma:int-det}(ii) and the continuous mapping
  theorem $Y_n\cap F \dodn Y\cap F$. The latter means that
  \begin{displaymath}
    \Prob{(Y_n\cap F)\cap L\neq\emptyset}\to \Prob{(Y\cap F)\cap L\neq\emptyset}
    \quad \text{as}\quad n\to\infty
  \end{displaymath}
  for all $L$ such that 
  \begin{displaymath}
    \Prob{(Y\cap F)\cap L\neq\emptyset}=\Prob{(Y\cap F)\cap \Int L\neq\emptyset}.
  \end{displaymath}
  By the definition of $Y_n$ and $Y$, this is the same as \eqref{eq:10}
  for $L$ satisfying \eqref{eq:108}.
\end{proof}

The next result follows either from Lemma~\ref{lemma:intersection} or
from Lemma~\ref{lemma:int-det} and continuous mapping theorem.

\begin{corollary}
  \label{cor:intersection}
  Let $X$ and $X_n$, $n\in\NN$, be random closed convex sets, whose interiors
  almost surely contain the origin. If $X_n\dodn X$ as $n\to\infty$,
  then $X_n\cap F\dodn X\cap F$ as $n\to\infty$ for each closed
  convex set $F$ which contains the origin.
\end{corollary}

The following result is used in the proof of Theorem~\ref{thm:main1}
in order to establish convergence in distribution of (not
necessarily convex) random closed sets by approximating them with
convex ones.

\begin{lemma}
  \label{lemma:approx}
  Let $(X_n)_{n\in\NN}$ be a sequence of random closed sets in
  $\R^d$. Assume that $Y_{m,n}^-\subset X_n\subset Y_{m,n}^+$ a.s.\
  for all $n,m\in\NN$ and sequences $(Y_{m,n}^-)_{n\in\NN}$ and
  $(Y_{m,n}^+)_{n\in\NN}$ of random closed sets. Further, assume that,
  for each $m\in\NN$:
  \begin{itemize}
  \item[(i)] the random closed set $Y_{m,n}^+$ converges in distribution to
    a random closed set $Y_m^+$ as $n\to\infty$;
  \item[(ii)] there exists a random closed set $Y_m^-$ such that
    \begin{equation}
      \label{eq:11}
      \Prob{Y_{m,n}^-\cap L\neq\emptyset,0\in Y_{m,n}^-}
      \to \Prob{Y_{m}^-\cap L\neq\emptyset,0\in Y_{m}^-}
      \quad \text{as}\quad n\to\infty
    \end{equation}
    for all $L\in\sC^d$ which are continuity sets for $Y_m^-$.
  \end{itemize}
  Further assume that $\Prob{0\in Y_m^-}\to 1$, and that $Y_m^+\downarrow Z$,
  $Y_m^-\uparrow Z$ a.s.\ as $m\to\infty$, in the Fell topology for a
  random closed set $Z$. Then $X_n\dodn Z$ as $n\to\infty$.
\end{lemma}
\begin{proof}
  Since the family of regular closed compact sets constitutes a
  convergence determining class, see Corollary~1.7.14 in~\cite{mo1},
  it suffices to check that the capacity functional of $X_n$ converges
  to the capacity functional of $Z$ on all compact sets $L$, which are
  regularly closed and are continuity sets for $Z$. There exist
  sequences $(L_k^-)$ and $(L_k^+)$ of compact sets, which are
  continuity sets for $Z$ and all $(Y_m^-)$ and $(Y_m^+)$,
  respectively, and such that $L_k^-\uparrow \Int L$ and
  $L_k^+\downarrow L$ as $k\to\infty$. These sets can be chosen from
  the families of inner and outer parallel sets to $L$, see page~148
  in~\cite{schn2}.

  Then
  \begin{align*}
    \Prob{Y_{m,n}^-\cap L_k^-\neq\emptyset,0\in Y_{m,n}^-}
    \leq \Prob{X_n\cap L\neq\emptyset}
    \leq \Prob{Y_{m,n}^+\cap L_k^+\neq\emptyset}.
  \end{align*}
  Passing to the limit as $n\to\infty$ yields that
  \begin{multline*}
    \Prob{Y_m^-\cap L_k^-\neq\emptyset,0\in Y_{m}^-}
    \leq \liminf_{n\to\infty} \,\Prob{X_n\cap L\neq\emptyset}\\
    \leq \limsup_{n\to\infty} \,\Prob{X_n\cap L\neq\emptyset}
    \leq \Prob{Y_m^+\cap L_k^+\neq\emptyset}.
  \end{multline*}
  Note that the a.s.\ convergence of $Y_m^\pm$ to $Z$ implies the
  convergence in distribution. Sending $m\to\infty$ and using that
  $\Prob{0\in Y_m^-}\to 1$, we conclude
  \begin{displaymath}
    \Prob{Z\cap L_k^-\neq\emptyset}
    \leq \liminf_{n\to\infty} \,\Prob{X_n\cap L\neq\emptyset}
    \leq \limsup_{n\to\infty} \,\Prob{X_n\cap L\neq\emptyset}
    \leq \Prob{Z\cap L_k^+\neq\emptyset}.
  \end{displaymath}
  Finally, sending $k\to\infty$ gives
  \begin{displaymath}
    \Prob{Z\cap \Int L \neq\emptyset}
    \leq \liminf_{n\to\infty} \,\Prob{X_n\cap L\neq\emptyset}
    \leq \limsup_{n\to\infty} \,\Prob{X_n\cap L\neq\emptyset}
    \leq \Prob{Z\cap L\neq\emptyset},
  \end{displaymath}  
  which completes the proof since
  $\Prob{Z\cap L\neq\emptyset}=\Prob{Z\cap \Int L\neq\emptyset}$.
\end{proof}

\begin{proposition}
  \label{prop:marked-sets}
  Let $\Psi_n:=\{(X_1,\xi_1),\dots,(X_n,\xi_n)\}$, $n\in\NN$, be a
  sequence of binomial point processes on
  $(\sK^d_0\setminus\{0\})\times\R^d$ obtained by taking independent
  copies of a pair $(X,\xi)$, where $X$ is a random closed convex set
  and $\xi$ is a random vector in $\R^d$, which may depend on
  $X$. Furthermore, let $\Psi:=\{(Y_i,y_i),i\geq1\}$ be a locally
  finite Poisson process on $(\sK^d_0\setminus\{0\})\times\R^d$ with
  the intensity measure $\mu$. Then
  $n^{-1}\Psi_n:=\{(n^{-1}X_i,\xi_i):i=1,\dots,n\}$ converges in
  distribution to $\Psi$ if and only if the following convergence
  takes place
  \begin{equation}
    \label{eq:10v}
    n\Prob{n^{-1}X\not\subset L,\xi\in B}=n\Prob{(n^{-1}X,\xi)\in \sA_L^c \times B}
    \to\mu(\sA_L^c \times B) \quad \text{as}\; n\to\infty,
  \end{equation}
  for every $\mu$-continuous set
  $\sA_L^c \times B\subset (\sK^d_0\setminus\{0\})\times \R^d$, where
  \begin{displaymath}
    \sA_L:=\{A\in\sK^d_0\setminus\{0\}: A\subset L\},
  \end{displaymath}
  and $L\in\sK^d_0\setminus\{0\}$ is an arbitrary compact convex set
  containing the origin and which is distinct from $\{0\}$.
\end{proposition}
\begin{proof}
  By a simple version of the Grigelionis theorem for binomial point
  processes (see, e.g., Proposition~11.1.IX in~\cite{dal:ver08} or
  Corollary~4.25 in~\cite{kalle17} or Theorem~4.2.5 in~\cite{mo1}),
  $n^{-1}\Psi_n\dodn\Psi$ if and only if
  \begin{equation}
    \label{eq:10v1}
    \mu_n(\sA\times B):= n\Prob{(n^{-1}X,\xi)\in \sA\times B}
    \to\mu(\sA\times B) \quad \text{as}\;n\to\infty,
  \end{equation}
  for all Borel $\sA$ in $\sK^d_0\setminus\{0\}$ and Borel $B$ in $\R^d$,
  such that $\sA\times B$ is a continuity set for $\mu$.

  Thus, we need to show that convergence \eqref{eq:10v1} follows from
  \eqref{eq:10v}. In other words, we need to show that the family of
  sets of the form $\sA_L^c \times B$ is a convergence determining
  class.

  Fix some $\eps>0$, and let $L_0:=B_\eps\subset \R^d$ be the closed
  centred ball of radius $\eps$. It is always possible to ensure that
  $\sA_{L_0}^c\times B$ is a continuity set for $\mu$. For each Borel
  $\sA$ in $\sK^d_0\setminus\{0\}$, put
  \begin{equation*}
    \label{eq:tildemu}
    \tilde{\mu}_n(\sA\times B)
    :=\frac{\mu_n\big((\sA\cap\sA_{L_0}^c)\times B\big)}
    {\mu_n\big(\sA_{L_0}^c\times \R^d\big)},\quad  n\geq 1,
  \end{equation*}
  and define $\tilde{\mu}$ by the same transformation applied to
  $\mu$.  Then $\tilde{\mu}_n$ is a probability measure on
  $(\sK^d_0\setminus\{0\})\times\R^d$, and so on $\sK^d\times\R^d$.
  Thus, $\mu_n$ defines the distribution of a random closed convex set
  $Z_n\times\{\zeta_n\}$ in $\sK^d\times\R^d$, which we can regard as
  a subset of $\sK^{d+1}$.

  Assume that we have shown that $\tilde{\mu}_n$ converges in distribution to
  $\tilde{\mu}$ as $n\to\infty$. Then \eqref{eq:10v} implies
  \eqref{eq:10v1}. Indeed, it obviously suffices to assume in
  \eqref{eq:10v1} that $\sA$ is closed in the Hausdorff metric and is
  such that $\sA\times B$ is a $\tilde{\mu}$-continuous
  set. Then there exists an $\eps>0$ such that each $A\in\sA$ is not a
  subset of $L_0=B_\eps$. Then $\sA\cap\sA_{L_0}^c=\sA$, so that
  \begin{displaymath}
    \frac{\mu_n(\sA\times B)}{\mu_n(\sA_{L_0}^c\times \R^d)}
    =\tilde{\mu}_n(\sA\times B)
    \;\;\to\;\;
    \tilde{\mu}(\sA\times B)
    =\frac{\mu(\sA\times B)}{\mu(\sA_{L_0}^c\times\R^d)}
    \quad \text{as}\; n\to\infty.
  \end{displaymath}
  Since the denominators also converge in view of \eqref{eq:10v} we
  obtain \eqref{eq:10v1}.

  In order to check that $\tilde{\mu}_n$ converges in distribution to
  $\tilde{\mu}$ we shall employ Theorem~1.8.14 from
  \cite{mo1}. According to the cited theorem $\tilde{\mu}_n$ converges
  in distribution to $\tilde{\mu}$ if and only if
  $\tilde{\mu}_n(\sA_L\times B)\to\tilde{\mu}(\sA_L\times B)$ for all
  $L\in\sK^d$ and compact convex $B$ in $\R^d$ such that
  $\sA_L\times B$ is a continuity set for $\tilde{\mu}$ and
  $\tilde{\mu}(\sA_L\times B)\uparrow 1$ if $L$ and $B$ increase to
  the whole space. The latter is clearly the case, since $\Psi$ has a
  locally finite intensity measure, hence, at most a finite number of
  its points intersects the complement of the centred ball $B_r$ in
  $\R^{d+1}$ for any $r>0$. For the former, note that, for every
  $L\in\sK^d\setminus \{0\}$,
  \begin{align*}
    \tilde{\mu}_n(\sA_L\times B)
    &=\frac{\mu_n\big(\sA_{L_0}^c\times B\big)
      -\mu_n\big((\sA^c_L\cap \sA_{L_0}^c)\times
      B\big)}{\mu_n\big(\sA_{L_0}^c\times \R^d\big)}
      =\frac{\mu_n\big((\sA^c_L\cup \sA_{L_0}^c)\times
      B\big)-\mu_n\big(\sA^c_L\times B\big)}
      {\mu_n\big(\sA_{L_0}^c\times \R^d\big)}\\
    &=\frac{\mu_n\big((\sA_L\cap \sA_{L_0})^c\times B\big)-\mu_n\big(\sA^c_L\times
      B \big)}{\mu_n(\sA_{L_0}^c\times \R^d\big)}
      =\frac{\mu_n\big(\sA^c_{L\cap L_0}\times B\big)-\mu_n\big(\sA^c_L\times B\big)}
      {\mu_n\big(\sA_{L_0}^c\times \R^d\big)}\\
    &\to \frac{\mu\big(\sA^c_{L\cap L_0}\times B\big)-\mu\big(\sA^c_L\times B\big)}
      {\mu\big(\sA_{L_0}^c\times \R^d\big)}
      =\tilde{\mu}(\sA_L\times B)\quad \text{as}\; n\to\infty,
   \end{align*}
   where the convergence in the last line follows
   from~\eqref{eq:10v}. The proof is complete.
 \end{proof}
 
\section*{Acknowledgement}

AM was supported by the National Research Foundation of Ukraine
(project 2020.02/0014 ``Asymptotic regimes of perturbed random walks:
on the edge of modern and classical probability'').  IM 
was supported by the Swiss Enlargement Contribution in the
framework of the Croatian--Swiss Research Programme (project number
IZHRZ0\_180549).

\let\oldaddcontentsline\addcontentsline% Store \addcontentsline
\renewcommand{\addcontentsline}[3]{}% Make \addcontentsline a no-op

% \bibliographystyle{abbrv}
% \bibliography{groups}

\let\addcontentsline\oldaddcontentsline% Restore \addcontentsline
\end{document}